\documentclass[12pt]{article}
\usepackage[total={6in, 9in}]{geometry}
\usepackage{mathrsfs}
\usepackage{t1enc}
\usepackage[latin1]{inputenc}
\usepackage{amsmath,amsthm,amssymb}
\usepackage{amsfonts}
\usepackage{latexsym}
\usepackage{enumerate}
\usepackage{dsfont}
\newcommand{\ve}{\varepsilon}
\usepackage{lineno}

\usepackage{graphicx}
\usepackage[natural]{xcolor}
\usepackage{tikz}
\usepackage{tikz-3dplot}
\usetikzlibrary{shapes}
\usetikzlibrary{arrows,decorations.pathmorphing,backgrounds,positioning,fit,petri,automata}
\usetikzlibrary {positioning}

\usepackage[
pdfauthor={},
pdftitle={},
pdfstartview=XYZ,
bookmarks=true,
colorlinks=true,
linkcolor=blue,
urlcolor=blue,
citecolor=blue,
bookmarks=true,
linktocpage=true,
hyperindex=true
]{hyperref}

\newtheoremstyle{nonitalic}  
{3pt}                      
{3pt}                      
{\normalfont}              
{}                         
{\bfseries}                
{.}                        
{ }                        
{}                         

\theoremstyle{nonitalic}
\newtheorem{thm}{Theorem}[section]

\newtheorem{lem}[thm]{Lemma}

\newtheorem{coro}[thm]{Corollary}
\newtheorem{conj}[thm]{Conjecture}
\newtheorem{claim}[thm]{Claim}

\newtheorem{LEM}[thm]{\textbf{Lemma}}
\theoremstyle{plain}

\theoremstyle{plain}

\theoremstyle{plain}

\newtheorem{fact}{Fact}
\DeclareMathOperator{\la}{la}
\DeclareMathOperator{\df}{def}

\begin{document}
\title{Linear arboricity of robust expanders}
\author{Yuping Gao\footnote{School of Mathematics and Statistics, Lanzhou University, Lanzhou 730000, China.}
	\qquad
	Songling Shan\footnote{Auburn University, Department of Mathematics and Statistics, Auburn, AL 36849, USA.
		}
}

\date{\today}
\maketitle

\vspace{-1cm}

\begin{abstract}  In 1980, Akiyama, Exoo, and  Harary conjectured that any graph $G$ can be decomposed into at most  $\lceil(\Delta(G)+1)/2\rceil$ linear forests. We confirm the conjecture for  robust expanders of linear minimum degree.
	As a consequence, the conjecture holds for dense quasirandom graphs of linear minimum degree
	as well as for large $n$-vertex graphs with minimum degree arbitrarily close to $n/2$ from above.

\medskip

\end{abstract}

\section{Introduction}
Unless otherwise stated explicitly, graphs in this paper are simple and finite.
 A \emph{linear forest} is a graph consisting of vertex-disjoint paths. The  \emph{linear arboricity} of a graph  $G$,
denoted $\la(G)$,  is the smallest number of linear forests its edges can be partitioned into.
This notion was introduced by Harary~\cite{MR0263677} in 1970 as one of the covering
invariants of graphs, and has been studied quite extensively since then.
The linear arboricity of any graph $G$ of maximum degree $\Delta$ is known to be at least
$\lceil \Delta/2\rceil$  as the edges incident with a maximum degree vertex need to be covered
by at least $\lceil \Delta/2\rceil$ distinct linear forests.
On the other hand,  $\la(G)$ is conjectured to be at most $\lceil (\Delta+1)/2\rceil$.
The  conjecture,
known as the \emph{Linear Arboricity Conjecture}, was proposed by Akiyama, Exoo and Harary~\cite{AEH1980} in 1980.
\begin{conj}[Linear Arboricity Conjecture]\label{laC}
	Every graph $G$ satisfies $\la(G) \le \lceil (\Delta(G)+1)/2\rceil$.
\end{conj}

The conjecture   would determine the linear arboricity exactly for graphs of odd maximum degree,  since in that case the upper and lower bounds coincide.  For graphs of even maximum degree it would imply that the linear arboricity must be one of only two possible values. However, determining the exact value among these two choices is NP-complete~\cite{MR0815871}.

In terms  of general bounds on $\la(G)$
for   graphs $G$ with maximum degree $\Delta$, the first result of this type was obtained by
Alon in 1988~\cite{MR0955135}, which states that $\la(G) \le \Delta/2+O\left(\frac{\Delta\log \log \Delta}{\log \Delta}\right)$.
The error term was improved to $O(\Delta^{2/3} \log^{1/3} \Delta)$ by Alon and Spencer in 1992~\cite{MR1140703},
and  has been further improved to $O(\Delta^{0.661})$ by Ferber, Fox and Jain~\cite{MR4074176},
and then to $O(\Delta^{1/2}\log^4\Delta)$ by Lang and Postle~\cite{MR4630467}.
Very recently,  Christoph, Dragani\'c, Gir\~ao, Hurley, Michel, and M\"uyesser~\cite{2507.20500} made a significant
improvement by reducing
the error term to $O(\log n)$, which is an exponential improvement over the previous best error term
if $\Delta =\Omega(n^\ve)$ for any given $0<\ve  \le 1$.
The Linear Arboricity Conjecture was well studied for classes of graphs including degenerate graphs, graphs with small maximum degree, and planar graphs.
For example, see~\cite{CH2021,CHKLW2012, EP1984,G1986,NZ2019,WWLC2014,W1999,YWS2022}.
The conjecture was confirmed for random graphs $G_{n,p}$ with constant  probability $p$ by Glock, K\"uhn, and Osthus~\cite{GKO2016} in 2016, and for the range
$\frac{C\log n}{n}\le p \le n^{-2/3}$ by Dragani\'c, Glock, Munh\'{a} Correia, and Sudakov in 2025, where $C$ is a constant~\cite{DGMS2025}. Glock, K\"uhn, and Osthus~\cite{GKO2016} also verified the conjecture for dense quasirandom graphs in which the difference between the maximum and minimum degrees is small.
Here, we use the following notion of quasirandomness, which is a one-sided version of $\varepsilon$-regularity. Let  $0<\ve, p<1$.
A graph $G$ on $n$ vertices is called \emph{lower-$(p,\ve)$-regular} if $e_G(S,T) \ge (p-\ve)|S||T|$  for all disjoint $S,T\subseteq V(G)$ with $|S|, |T| \ge \ve n$.

The Linear Arboricity Conjecture holds for regular robust expanders of linear minimum degree, which we define now.
Given $0 <\nu \leq \tau< 1$, we say that a graph $G$ on $n$ vertices is a \emph{robust $(\nu, \tau)$-expander},
if for all $S\subseteq V (G)$ with $\tau n\leq |S|\leq (1-\tau)n$ the number of vertices that have at least $\nu n$
neighbors in $S$ is at least $|S| + \nu n$. The \emph{$\nu$-robust neighborhood} $RN_{G}(S)$ is the set of all
those vertices of $G$ which have at least $\nu n$ neighbors in $S$.
Note that for
$0 < \nu' \le \nu \le \tau \le \tau'<1$, any robust $(\nu,\tau)$-expander is also a robust $(\nu',\tau')$-expander.
The theorem  below by K\"uhn and Osthus~\cite{KO2014}
is   profound and also is a powerful tool that has been used to establish other results.

\begin{thm}[\cite{KO2014}]\label{thm2.13} For every $\alpha>0$ there exists $\tau=\tau(\alpha)> 0$ such that for every $\nu> 0$ there
	exists an integer $N_0=N_0(\alpha, \nu, \tau)$ for which the following holds. Suppose that
	\begin{enumerate}[\rm{(i)}]
		\item [\rm{(i)}] $G$ is an $r$-regular graph on $n\geq N_0$ vertices, where $r\geq \alpha n$ is even;
		
		\item [\rm{(ii)}] $G$ is a robust $(\nu, \tau)$-expander.
	\end{enumerate}
	Then $G$ has a Hamilton decomposition. Moreover, this decomposition can be found in
	polynomial time in $n$.
\end{thm}

As an application of Theorem~\ref{thm2.13}, the Linear Arboricity Conjecture holds for regular robust expanders of linear minimum degree,
as stated below. The proof follows exactly the same proof of Theorem~\ref{JCTB2016} in~\cite{GKO2016}, and we omit  it  here.

\begin{thm} \label{cor:regular-expander}  Let $n,r\in \mathbb{N}$ and suppose that $0<1/n \ll \nu  \le  \tau,  \alpha <1$.
	If $G$ is a robust $r$-regular  $(\nu,\tau)$-expander on $n$ vertices with $r \ge \alpha n$, then   $\la(G)\leq \lceil\frac{r+1}{2}\rceil$.
\end{thm}

Using Theorem~\ref{thm2.13},
Glock, K\"uhn, and Osthus~\cite{GKO2016} in 2016 also proved the Linear Arboricity Conjecture
for regular graphs of large degree.

\begin{thm}[\cite{GKO2016}]\label{JCTB2016}
	There exists  $n_0\in \mathbb{N}$   such that every $d$-regular graph $G$ on $n \ge n_0$ vertices with $d\geq \left\lfloor\frac{n-1}{2}\right\rfloor$ has a decomposition into $\left\lceil\frac{d+1}{2}\right\rceil$ linear forests.
\end{thm}

In this paper, we investigate the  Linear Arboricity Conjecture   for robust expanders of
linear minimum degree that are not necessarily regular and  obtain the following result.

\begin{thm}\label{thm:expander}
	For every $\alpha>0$ there exists $\tau=\tau(\alpha)> 0$ such that for every $\nu> 0$ there
	exists an integer $N_0=N_0(\alpha, \nu, \tau)$ for which the following holds. If
	$G$ is a   robust $(\nu, \tau)$-expander on $n\geq N_0$ vertices with $\delta(G)\geq \alpha n$,
	then  $\la(G)\leq \lceil\frac{\Delta(G)+1}{2}\rceil$.
\end{thm}

Note that Conjecture~\ref{laC} is equivalent to its restriction on regular graphs, as any   graph $G$
with maximum degree $\Delta$ can be embedded in a   $\Delta$-regular graph $H$.  However,
Theorem~\ref{thm:expander} does not follow from Theorem~\ref{cor:regular-expander}.
For example,  for sufficiently large $n$, we let $G$ be obtained from $K_{n-3}$ and $K_3$
by adding edges  joining each vertex of the  $K_3$ and  about $(n-3)/3$ distinct vertices of the $K_{n-3}$
such that the resulting graph has  maximum degree $n-3$. It is clear that the graph is a robust expander.
However, any $(n-3)$-regular supergraph $H$ of $G$ has at least $n+(n-3-\frac{n+3}{3})$ vertices and has a cut of
size 3. The graph is not a   robust $(\nu,\tau)$-expander for any $0<\nu \le \tau \le 2/5$.

It was shown independently by Glock, K\"uhn, and Osthus~\cite{GKO2016}  and
 K\"uhn and Osthus~\cite{KO2014} that dense quasirandom graphs and graphs with large minimum degree are robust expanders.

\begin{lem}[\cite{GKO2016}]\label{lem:quasi-expander}
Let $0<1/n_0 \ll \ve \ll \tau, p <1$ and suppose $G$ is a lower-$(p,\ve)$-regular graph on $n\ge n_0$ vertices.
Then $G$ is a robust $(\ve, \tau)$-expander.
\end{lem}

\begin{lem}[\cite{KO2014}]\label{expander}  Suppose that $0<\nu \leq \tau \leq \ve<1$ are such that $\varepsilon\geq 4\nu/\tau$. Let $G$ be a graph on $n$ vertices with minimum degree $\delta(G)\geq(1+\ve)n/2$. Then $G$ is a robust $(\nu, \tau)$-expander.
\end{lem}

As a consequence of Theorem~\ref{thm:expander} and Lemmas~\ref{lem:quasi-expander} and~\ref{expander}, we have the following corollaries.

\begin{coro}\label{cor1} For any given $0<\varepsilon, p, \alpha<1$, there  exists $n_0\in \mathbb{N}$  for which  the following statement holds.   If $G$ is a lower-$(p,\ve)$-regular graph on $n\ge n_0$ vertices with $\delta(G) \ge \alpha n$, then $\la(G)\leq \lceil\frac{\Delta(G)+1}{2}\rceil$.
\end{coro}

\begin{coro}\label{cor2} For any given $0<\varepsilon<1$, there  exists $n_0\in \mathbb{N}$  for which  the following statement holds. If $G$ is a graph on $n\geq n_0$ vertices with $\delta(G)\geq \frac{1}{2}(1+\ve)n$, then $\la(G)\leq \lceil\frac{\Delta(G)+1}{2}\rceil$.
\end{coro}

Note again that
Corollary~\ref{cor2} does not follow from Theorem~\ref{JCTB2016} as passing to a regular supergraph may significantly increase the number
of vertices. Indeed, we now  construct a graph $G$ of maximum degree $\Delta$  such that any $\Delta$-regular supergraph $H$ of $G$ satisfies $\Delta< |V(H)|/2$.
Choose  $n$  such that both  $\frac{1}{2} n$ and  $\ve n$ are even integers. We let $G_1$  be a $3\ve n$-regular graph on $\frac{1}{2} n-\ve n$  vertices,
$G_2$  be a $5\ve n$-regular graph on  another $\frac{1}{2} n-\ve n$  vertices, and  $G_3$ be another complete graph on $2\ve n$  vertices.
Then the graph $G$ is obtained from $G_1, G_2, G_3$ by adding all edges between $G_1$ and $G_2$, and all edges between $G_1$ and $G_3$.
Now  we have $\delta(G)=\frac{1}{2} n-\ve n+2\ve n -1=\frac{1}{2} n+\ve n-1$, and $\Delta:=\Delta(G) =\frac{1}{2} n-\ve n+3\ve n +2\ve n =\frac{1}{2} n+4\ve n$.
Also, all vertices of $V(G_1) \cup V(G_2)$ are of maximum degree in $G$. Thus
we have $\sum_{v\in  V(G)} (\Delta -d_G(v)) =2\ve n (3\ve n+1)$. Since all vertices of minimum degree in $G$ form a clique, and  $\Delta(G)-\delta(G) =3\ve n +1$,    it follows that  any  $\Delta$-regular supergraph  $H$ of $G$ has at least $n+3\ve n+1$  vertices.
On average, one vertex  $v$ from  $V(H)\setminus V(G)$  is adjacent in $H$ to at most $2\ve n$
vertices from $V(G)$.  As $H$ is $\Delta$-regular,  $v$ has at least $\Delta -2\ve n$ neighbors in $H$ from $V(H)\setminus V(G)$.
Thus   $H$  has more than   $n+ \Delta-2\ve n >\frac{3}{2} n $ vertices. As a consequence, we have
$\Delta  <\frac{1}{3} (1+8\ve )|V(H)|$,  where  $ \frac{1}{3} (1+8\ve )|V(H)| < \frac{1}{2}|V(H)|$ if $\ve <\frac{1}{16}$.

Our proof of Theorem~\ref{thm:expander}, which relies on Theorem~\ref{thm2.13}, will be given in Section~3 after some preliminaries in the next section.

\section{Notation and Preliminaries}

Let $G$ be a graph and $S\subseteq V(G)$.  We let  $n(G) =|V(G)|$ and $e(G)=|E(G)|$.
The subgraphs of $G$ induced on $S$ and $G-S$  are  respectively denoted by $G[S]$ and  $G-S$.
For $v\in V(G)$, we write $G-v$ for $G-\{v\}$.
For two  disjoint subsets $S,T$ of $V(G)$, we denote by $E_{G}(S,T)$  the set of edges in $G$ with one endvertex in $S$ and the other in $T$, and   let $e_G(S,T)=|E_G(S,T)|$.  If $F\subseteq E(G)$,  then $G-F$ is obtained from $G$ by deleting all edges in $F$.
If $F$ is a submultiset of edges from the complement $\overline{G}$, then $G+F$
is obtained from $G$ by adding the edges of $F$ to $G$ ($G+F$ may become a multigraph).  For two integers $p$ and $q$, let $[p, q] = \{i\in \mathbb{Z}\mid p\leq i\leq q\}$.

We will make use of the following notation: $0<a\ll b\leq1$. Precisely,
if we claim that a statement is true whenever $0<a\ll b\leq1$, then this means that
there exists a non-decreasing function $f : (0, 1]\rightarrow (0, 1]$ such that the statement holds
for all $0<a\ll b\leq1$ satisfying $a\leq f(b)$. Intuitively, this means that for every $b>0$, the statement is true provided that $a$ is sufficiently small compared to $b$.

\subsection{Degree sequences and number of maximum degree vertices}

Let $n\geq 2$ be an integer. A sequence of nonnegative integers listed in non-increasing order
$(d_1,\ldots, d_n)$ is \emph{graphic} if there is a graph $G$ such that the degree sequence of $G$ is $(d_1,\ldots, d_n)$. In this case, we say that $G$ \emph{realizes} $(d_1,\ldots, d_n)$.

The first result below by Hakimi~\cite{H1962} gives a sufficient and necessary condition
for a sequence of integers to be the degree sequence of a multigraph.
The second result by the second author~\cite{S2023} will be used to construct a supergraph of a given graph $G$ by
adding additional vertices.

\begin{thm}[\cite{H1962}]\label{thm:degree-sequence-multigraph} Let $0\leq d_n\leq \ldots \leq d_1$ be integers. Then there exists a multigraph $G$ on vertices $v_1,v_2,\ldots,v_n$ such that $d_{G}(v_i)=d_i$ for all $i\in [1,n]$ if and only if $\sum\limits_{i=1}^{n}d_i$ is even and $\sum\limits_{i=2}^{n}d_i\geq d_1$.
\end{thm}

\begin{lem}[\cite{S2023}]\label{graphic}Let $n$ and $d$ be positive integers such that $n\geq d+1\geq 3$. Suppose $(d_1,\ldots, d_n)$ is a sequence of positive integers with $d_1=\ldots=d_t=d$ and $d_{t+1}=\ldots=d_n = d-1$ and $\sum\limits_{i=1}^{n}d_i$ even, where $t\in [1,n]$. Then $(d_1,\ldots, d_n)$ is graphic. Furthermore, the sequence
can be realized in polynomial time in $n$.
\end{lem}

We will also need the following result by Plantholt and the second author~\cite{PS2023} which gives a lower bound on the number of maximum degree vertices
if we assume that all vertices of degree less than maximum degree form a clique.

\begin{lem}[\cite{PS2023}]\label{VDelta}  Let $G$ be a graph on $n$ vertices such that all vertices of degree less than $\Delta(G)$ are mutually adjacent in $G$. Then  $G$ has more than $n/2$ vertices of maximum degree.
\end{lem}

\subsection{Matchings in almost-regular graphs}

Let $G$ be a graph.
A \emph{matching}  $M$ in  $G$ is a set of vertex-disjoint edges.  The set of vertices \emph{covered} or \emph{saturated} by $M$ is denoted by $V(M)$.
We  say  that $M$ is \emph{perfect}    if $V(M)=V(G)$.
The graph $G$ is \emph{factor-critical} if $G$ does not have a perfect matching but  $G-v$ has one
for any $v\in V(G)$.

For a subset $X\subseteq V(G)$, we denote by $o(G-X)$
the number of components of $G-X$
of odd order.
Define  ${\rm df}(G)=\max_{X\subseteq V(G)} (o(G-X)-|X|)$.
By applying Tutte's Theorem~\cite{MR0023048}
to the graph obtained from $G$ by adding ${\rm df}(G)$ vertices and joining all edges
between $V(G)$ and these new vertices, Berge~\cite{MR0100850} observed that
the maximum size of a matching  in $G$ is $\frac{1}{2}(n(G)-{\rm df}(G))$.
West~\cite{MR2788782} gave a short proof of this result. To state his result,
we need some more definitions.

For $X\subseteq V(G)$,  denote by $Y$  the set of components of $G-X$.
We define an auxiliary bipartite multigraph $B(X)$ by contracting each component of $G-X$
to a single vertex and deleting edges within $X$.
Thus  for any  vertex $x\in X$ and any $y\in V(B(X))\setminus X$ such that $y$ is contracted from
a component $D_y$ of $Y$, there are exactly $e_G(\{x\}, V(D_y))$ edges between $x$ and $y$ in $B(X)$.

\begin{LEM}\label{lem:matching-structure}
	Let $G$ be a multigraph and $X\subseteq V(G)$ be maximal with $o(G-X)-|X|={\rm df}(G)$. Then
	the following statements hold.
	\begin{enumerate}[(1)]
		\item (\cite[Lemma 2]{MR2788782}) Every component of $G-X$ has odd order and is factor-critical.
		\item  (\cite[Lemma 3]{MR2788782}) The bipartite multigraph $B(X)$ has a matching covering $X$.
		\item  The bipartite multigraph $B(X)$ has a matching covering $X$ and all vertices of degree at least $d$
	from  $V(B(X))\setminus X$ in $B(X)$, where $d=\max\{d_{B(X)}(x)\mid x\in X\}$.
	\end{enumerate}
\end{LEM}

\begin{proof} Statements (1)-(2) were proved by West in~\cite{MR2788782}.

We prove statement (3). By (2), $B(X)$ has a matching $M_1$ covering $X$. Let $Y=V(B(X))\setminus X$, $d=\max\{d_{B(X)}(x) \mid x\in X\}$, $Y_d=\{y\in Y \mid d_{B(X)}(y) \ge d\}$. By applying Hall's Theorem to $B(X)[X,Y_d]$, $B(X)[X,Y_d]$ and so $B(X)$ has a matching $M_2$ saturating all vertices from $Y_d$, where $B(X)[X,Y_d]$ is the induced  subgraph of $B(X)$ with bipartitions as $X$ and $Y_d$.
In each component  $D$ of the subgraph of $B(X)$ induced on  the symmetric difference of $M_1$ and $M_2$,  there is  a matching that saturates vertices in $V(D)\cap (X\cup Y_d)$. The union of these matchings together with edges of $M_1\cap M_2$ gives a desired matching that covers all vertices of $X\cup Y_d$.
\end{proof}

A graph $G$ is \emph{almost $r$-regular} for some integer $r\ge 0$ if exactly one vertex has degree $r+1$ and all others have degree $r$.

\begin{LEM}\label{lem:matching-in-complement}
	Let $G$ be an   $n$-vertex  almost $r$-regular  graph  for some integer $r\ge 0$. Then $ \overline{G}$ has a matching covering
	at least   $ n-\frac{n}{n-r}-3$ vertices of $\overline{G}$ excluding the vertex of degree $r+1$ in $G$.
\end{LEM}

\proof Let $x\in V(G)$ be the vertex with $d_G(x)=r+1$.
Note that  $ \overline{G}$ has $n-1$ vertices of degree $(n-1-r)$ and one vertex $x$ of degree $n-2-r$. If  $ \overline{G}$ has a perfect matching, then we are done.
Thus assume that $\overline{G}$ does not have a perfect matching. Let $X\subseteq V(\overline{G})$ be maximal
such that $o(\overline{G}-X)-|X|={\rm df}(\overline{G})$.    Then $\overline{G}-X$ has no even component by
Lemma~\ref{lem:matching-structure}(1).
For any odd component $D$ of $\overline{G}-X$ such that $x\not\in V(D)$ and $|V(D)| \le n-r-1$, we have
\begin{eqnarray*}
	e_{\overline{G}}( V(D),X) &\ge&  |V(D)|(n-r-1-(|V(D)|-1))  \\
	&\ge & n-r-1,
\end{eqnarray*}
where the last inequality is obtained as the quadratic function $ |V(D)|(n-r-1-(|V(D)|-1))$ in $|V(D)|$
achieves its minimum at $|V(D)|=1$ or $|V(D)|=n-r-1$.
There are at most  $\frac{n}{n-r}$   odd components  $D$ of $\overline{G}-X$ with $|V(D)|  \ge n-r$,
and so there are at most $ \frac{n}{n-r}+1$ odd components $D$ of $\overline{G}-X$
such that $e_{\overline{G}}(V(D), X) \le n-r-2$.   As every vertex of $X$ has degree at most $n-1-r$
in $\overline{G}$,
by Lemma~\ref{lem:matching-structure}(3), we know
that $o(\overline{G}-X)-|X| \le \frac{n}{n-r}+1$ and so
${\rm df}(\overline{G}) \le  \frac{n}{n-r}+1$.
Thus by Berge's formula  from~\cite{MR0100850} that
the maximum size of a matching  in a graph $H$ is $\frac{1}{2}(n(H)-{\rm df}(H))$, we  conclude that
$\overline{G}$ has a matching that covers at least $n-\frac{n}{n-r}-1$ vertices of $\overline{G}$,
and so it has a matching that covers at least $n-\frac{n}{n-r}-3$ vertices of $\overline{G}$ excluding $x$.
\qed

\subsection{Robust expanders}

The following lemma states that robust expansion is preserved when few edges are removed and/or few vertices are removed and/or added.
The result is stated for robust outexpanders in~\cite{MR4550147}, we list here its undirected counterpart.

\begin{LEM}[{\cite[Lemma 4.2]{MR4550147}}]\label{lem:stability-expansion}
	Let $0<\varepsilon\le \nu\le \tau\le 1$. Let $G$ be a robust $(\nu,\tau)$-expander on $n$ vertices.
	\begin{enumerate}[(a)]
		\item If $G'$ is obtained from $G$ by removing at most $\varepsilon n$ edges at each vertex, then $G'$ is a robust $(\nu-\varepsilon,\tau)$-expander.
		\item Suppose that $\tau \ge (1+2\tau)\varepsilon$. If $G'$ is obtained from $G$ by adding or removing at most $\varepsilon n$ vertices, then $G'$ is a robust $(\nu-\varepsilon,2\tau)$-expander.
	\end{enumerate}
\end{LEM}

Following~\cite{MR3299598}, for a given integer $k\ge 2$, we say a graph $G$ is \emph{Hamilton $k$-linked}
if whenever $x_1,y_1, \ldots, x_k, y_k$ are distinct vertices, there exist vertex-disjoint paths $P_1, \ldots, P_k$
such that $P_i$ connects $x_i, y_i$ for each $i\in [1,k]$ and such that together the paths $P_1, \ldots, P_k$ cover all vertices of $G$.
We will use the following result by K\"uhn, Lo,   Osthus,  and Staden from~\cite{MR3299598}.

\begin{lem}[{\cite[Corollary 6.9(ii)]{MR3299598}}]\label{lem:hamiltonicity-of-expander}
	Let $n,k\in \mathbb{N}$ and suppose that $0<1/n \ll \nu \ll \tau  \ll \alpha <1$ and $k \le \nu^4 n$.
	If $H$ is a robust $(\nu,\tau)$-expander on $n$ vertices with $\delta(H) \ge \alpha n$, then $H$ is Hamilton $k$-linked.
\end{lem}

We next introduce the notion of a robust outexpander and cite a  theorem of  Gir\~ao, Granet, K\"uhn, Lo, and Osthus~\cite{MR4550147}
 on  ``almost-regular'' graphs. The result will help us with linear forest decomposition when the graph is close to regular.

Let  $D$ be  a digraph on $n$ vertices.  We write $xy$ for an edge which is directed from $x$ to $y$, and call $y$ an \emph{outneighbor} of $x$ and $x$ an \emph{inneighbor} of $y$.
For $x\in V(D)$,  the  outdegree of $x$, denoted  $d^+_D(x)$, is the number of edges  $xy$ for all $y\in V(D)$;
and the indegree of $x$, denoted $d^-_D(x)$, is the number of edges  $yx$ for all $y\in V(D)$.
Define the minimum \emph{semidegree} of $D$ as $\delta^0(D) =\min\{d_D^+(x), d_D^-(x) \mid  x\in V(D)\}$.

Let $0 < \nu \le \tau < 1$ and $S \subseteq V(D)$.  The \emph{$\nu$-robust outneighborhood} $RN^+_{\nu,D}(S)$ of $S$ is the set of all those vertices $x$ of $D$ which have at least $\nu n$ inneighbors  from  $S$.  The digraph $D$ is called a \emph{robust $(\nu,\tau)$-outexpander} if
\[
|RN^+_{\nu,D}(S)| \ge |S| + \nu n
\]
for all $S \subseteq V(D)$ with $\tau n \le |S| \le (1-\tau)n$.

Let $\varepsilon, \delta > 0$. We say $D$ is \emph{$(\delta,\varepsilon)$-almost regular} if, for each $v \in V(D)$,  $d^+_D(v) = (\delta \pm \varepsilon) n$ and $d^-_D(v) = (\delta \pm \varepsilon) n$. Similarly,  if $D$ is a graph, then it is
$(\delta,\varepsilon)$-\emph{almost regular} if, for each $v \in V(D)$, it holds that $d_D(v) = (\delta \pm \varepsilon) n$.

Let $0<\ve, p<1$.
We say $D$ is an \emph{$(\varepsilon,p)$-robust $(\nu,\tau)$-outexpander} if $D$ is a robust $(\nu,\tau)$-outexpander and, for any integer $k \ge \varepsilon n$, if $S \subseteq V(D)$ is a random subset of size $k$, then $D[S]$ is a robust $(\nu,\tau)$-outexpander with probability at least $1-p$.  By the definitions, we have the following fact.

\begin{fact}\label{fact:robust}
	Let $D$ be an $(\varepsilon,p)$-robust $(\nu,\tau)$-outexpander. Then, for any $\varepsilon' \ge \varepsilon$, $p' \ge p$, $\nu' \le \nu$, and $\tau' \ge \tau$, $D$ is an $(\varepsilon',p')$-robust $(\nu',\tau')$-outexpander.
\end{fact}

The lemma below says that  an   induced subgraph on a subset  with relatively large size of a robust outexpander  is also a robust outexpander.

\begin{LEM}[{\cite[Lemma 14.3]{MR4550147}}]\label{lem:robustscale}
	Let $0 < \frac{1}{n} \ll \varepsilon \ll \nu' \ll  \nu, \tau,\alpha \ll 1$. Suppose $D$ is a robust $(\nu,\tau)$-outexpander on $n$ vertices satisfying $\delta^0(D) \ge \alpha n$. Then $D$ is an $(\varepsilon, n^{-2})$-robust $(\nu',4\tau)$-outexpander.
\end{LEM}

We also need the following result  by K\"uhn and Osthus.

\begin{lem}[{\cite[Lemma 3.2]{KO2014}}]\label{expander-split}  Suppose that $0 < 1/n \ll \eta \ll \nu, \tau, \alpha, \lambda, 1-\lambda < 1$.
	Let $G$ be a digraph on $n$ vertices with $\delta^0(G) \ge \alpha n$ which is a
	robust $(\nu,\tau)$-outexpander. Then $G$ can be split into two edge-disjoint
	spanning subdigraphs $G_1$ and $G_2$ such that the following two properties hold.
	\begin{enumerate}[(i)]
		\item
		$d^+_{G_1}(x) = (1 \pm \eta)\lambda d^+_G(x)$ and $d^-_{G_1}(x) = (1 \pm \eta)\lambda d^-_G(x)$ for every $x \in V(G)$.
		\item $G_1$ is a robust $(\lambda \nu/2,\tau)$-outexpander and $G_2$ is a robust $((1-\lambda)\nu/2,\tau)$-outexpander.
	\end{enumerate}
\end{lem}

Let $D$ be a  multidigraph  on vertex set  $V$. Let $L$ be  a multiset consisting of paths on $V$ and isolated vertices.
We use $V(L)$ to denote the set of vertices that are  either isolated in $L$  or on the non-trivial  paths contained in $L$,
 use $E(L)$ for the multiset of edges that are contained in paths of $L$,
 and for $v\in V(L)$, we use $d_L(v)$ to denote the number of edges from $E(L)$ that have $v$
 as one endvertex.   Let $F$ be a multiset of edges.
We say $(L,F)$ is a \emph{layout} if $F \subseteq E(L)$ and $E(L) \setminus F \neq \emptyset$.
We define a layout for a multigraph exactly the same way as above.  When $F$ is a multidigraph or multigraph,
we also simply say that $(L,F)$ is a layout instead of $(L,E(F))$ being a layout.

Given an edge $uv$, we say that a  $(u,v)$-path  has  \emph{shape $uv$}.
Let $(L,F)$ be a layout on $V$. A multidigraph $\mathcal{H}$ on $V$ is a \emph{spanning configuration of shape $(L,F)$} if $\mathcal{H}$ can be decomposed into internally vertex-disjoint paths $\{P_e \mid e \in E(L)\}$ such that each $P_e$ has shape $e$; $P_f = f$ for all $f \in F$; and $\bigcup_{e \in E(L)} V^0(P_e) = V \setminus V(L)$ (given a path $P$, $V^0(P)$ denotes the set of internal vertices of $P$). (Note that the last equality implies that the isolated vertices of $L$ remain isolated in $\mathcal{H}$.)
A spanning configuration of shape $(L,F)$ for a  multigraph is defined the same way as above.
We will need the following result by  Gir\~ao, Granet, K\"uhn, Lo, and Osthus.

\begin{LEM}[{\cite[Lemma 7.3]{MR4550147}}]\label{lem:7.3}
	Let $0 < \frac{1}{n} \ll \varepsilon \ll \nu \ll \tau \ll \gamma \ll \eta, \delta \le 1$. Suppose $\ell \in \mathbb{N}$ satisfies $\ell \le (\delta - \eta)n$. If $\ell \le \varepsilon^2 n$, then let $p \le n^{-1}$; otherwise, let $p \le n^{-2}$. Let $D$ and $\Gamma$ be edge-disjoint digraphs on a common vertex set $V$ of size $n$. Suppose that $D$ is $(\delta,\varepsilon)$-almost regular and $\Gamma$ is $(\gamma,\varepsilon)$-almost regular. Suppose further that $\Gamma$ is an $(\varepsilon,p)$-robust $(\nu,\tau)$-outexpander. Let $\mathcal{F}$ be a multiset of directed edges on $V$. Any edge in $\mathcal{F}$ is considered to be distinct from the edges of $D \cup \Gamma$, even if the starting and ending points are the same. Let $F_1,\dots,F_\ell$ be a partition of $\mathcal{F}$. Assume that $(L_1,F_1),\dots,(L_\ell,F_\ell)$ are layouts such that $V(L_i) \subseteq V$ for each $i \in [1,\ell]$ and the following hold, where $L := \bigcup_{i \in [1,\ell]} L_i$.
	\begin{enumerate}[(a)]
		\item For each $i \in [1,\ell]$, $|V(L_i)| \le \varepsilon^2 n$ and $|E(L_i)| \le \varepsilon^4 n$.
		\item Moreover, for each $v \in V$, $d_L(v) \le \varepsilon^3 n$ and there exist at most $\varepsilon^2 n$ indices $i \in [1,\ell]$ such that $v \in V(L_i)$.
	\end{enumerate}
	Then there exist edge-disjoint submultidigraphs $\mathcal{H}_1,\dots,\mathcal{H}_\ell \subseteq D \cup \Gamma \cup \mathcal{F}$ such that, for each $i \in [1,\ell]$, $\mathcal{H}_i$ is a spanning configuration of shape $(L_i,F_i)$ and  that $D\cup \Gamma -\bigcup_{i \in [1,\ell]} E(\mathcal{H}_i)$ is a robust $(\nu/2,\tau)$-outexpander.
\end{LEM}

Results on  robust outexpanders can be modified for robust expanders. We will apply the result below
and adapt Lemma~\ref{lem:7.3} for robust expanders.

\begin{LEM}[{\cite[Lemma 3.1]{KO2014}}]\label{lem:orient}
	Suppose that $0 < 1/n \ll \eta' \ll \nu, \tau, \alpha < 1$, and  that $G$ is a robust $(\nu,\tau)$-expander on $n$ vertices with $\delta(G) \ge \alpha n$. Then one can orient the edges of $G$ in such a way that the oriented digraph $D$ thus obtained from $G$ satisfies the following properties:
	\begin{itemize}
		\item[(i)] $D$ is a robust $(\nu/4, \tau)$-outexpander.
		\item[(ii)] $d^+_{D}(x) = (1 \pm \eta') d_G(x)/2$ and $d^-_{D}(x) = (1 \pm \eta') d_G(x)/2$ for every vertex $x$ of $D$.
	\end{itemize}
\end{LEM}

We now adapt Lemma~\ref{lem:7.3} for robust expanders.

 \begin{thm}\label{thm:hamilton2}
 Let $0 < \frac{1}{n}  \ll \ve^* \le  \ve \ll   \nu' \ll \nu \ll \tau  \ll \gamma \ll   \alpha,\eta^* \ll 1$, and let
$\ell \le \frac{1}{2}(\alpha - \eta^*)n$. Suppose $G$ is an $(\alpha,\ve^*)$-almost regular  robust $(\nu,\tau)$-expander on $n$ vertices. Suppose that, for each $i \in [1,\ell]$, $F_i$ is a multiset of edges on $V(G)$.
 Assume that $(L_1,F_1),\dots,(L_\ell,F_\ell)$ are layouts such that $V(L_i) \subseteq V(G)$ for each $i \in [1,\ell]$ and the following hold, where $L := \bigcup_{i \in [1,\ell]} L_i$.
 \begin{enumerate}[(a)]
 	\item For each $i \in [1,\ell]$, $|V(L_i)| \le \varepsilon^2 n$ and $|E(L_i)| \le \varepsilon^4 n$.
 	\item Moreover, for each $v \in V(G)$, $d_L(v) \le \varepsilon^3 n$ and there exist at most $\varepsilon^2 n$ indices $i \in [1,\ell]$ such that $v \in V(L_i)$.
 \end{enumerate}
 Then there exist edge-disjoint  submultigraphs   $\mathcal{H}_1,\ldots,\mathcal{H}_{\ell} \subseteq G + \bigcup_{i \in [1,\ell]} F_i$ such that each $\mathcal{H}_i$ is a spanning configuration of shape $(L_i,F_i)$ and that
 $G-\bigcup_{i=1}^\ell E(\mathcal{H}_i)$
 is a robust $(\nu'/2,4\tau)$-expander.
 \end{thm}

\proof  Let  $\eta'$ be defined as  in Lemma~\ref{lem:orient} such that   $1/n \ll \eta' \ll \ve^*$.
By Lemma~\ref{lem:orient}, we can orient the edges of $G$ to get a digraph $D$
such that $D$ is a robust $(\nu/4, \tau)$-outexpander with
$d^+_{D}(x) = (1 \pm \eta') d_G(x)/2$ and $d^-_{D}(x) = (1 \pm \eta') d_G(x)/2$ for every vertex $x$ of $D$.
Since $G$ is $(\alpha,\ve^*)$-almost regular, it follows that
$d^+_{D}(x), d^-_{D}(x)= (\frac{1}{2} \alpha \pm \frac{1}{2}(\ve^*+\eta')) n$. Thus $D$ is  $(\frac{1}{2} \alpha, \frac{1}{2}(\ve^*+\eta'))$-almost regular.
As  $\ve^* \le  \ve$,  $D$  is also $(\frac{1}{2} \alpha, \ve)$-almost regular.

Taking  $\lambda =\frac{2\gamma}{\alpha}$ and applying
Lemma~\ref{expander-split},  we split $D$
into two edge-disjoint spanning subdigraphs $\Gamma$ and $D'$
such that $\Gamma$ is  a  $(\gamma,\ve)$-almost regular
  robust $(\gamma \nu/(4\alpha),\tau)$-outexpander,
  and $D'$ is a $(\alpha/2-\gamma, \ve)$-almost regular robust
  $((\alpha-2\gamma)\nu/(8\alpha),\tau)$-outexpander.
  As $\nu' \ll \nu \ll \gamma \ll \alpha$ and so  $\nu'$ can be chosen so that $\nu' \ll \gamma \nu/(4\alpha)$,
  by Lemma~\ref{lem:robustscale},
  $\Gamma$ is an $(\ve, n^{-2})$-robust $(\nu',4\tau)$-outexpander.

  Applying Lemma~\ref{lem:7.3} with $D'$, $\alpha/2-\gamma$, $\nu'$, $\varepsilon^{1/5}$, and $\eta^*/2$ playing the roles of $D,\delta,\nu,\varepsilon$, and $\eta$,
we find edge-disjoint submultidigraphs $\mathcal{H}_1,\dots,\mathcal{H}_\ell \subseteq D' \cup \Gamma + \bigcup_{i \in [1,\ell]} F_i$ such that, for each $i \in [1,\ell]$, $\mathcal{H}_i$ is a spanning configuration of shape $(L_i,F_i)$ and $D' \cup \Gamma-\bigcup_{i=1}^\ell E(\mathcal{H}_i)$ is a robust $(\nu'/2,4\tau)$-outexpander.  The latter implies that $G-\bigcup_{i=1}^\ell E(\mathcal{H}_i)$
is a robust $(\nu'/2,4\tau)$-expander.
\qed

\section{Proof of Theorem~\ref{thm:expander}}

For a graph $G$ and an integer $i\ge 0$, define $V_i(G)=\{v\in V(G)\mid d_G(v)=i\}$,  $W(G)=V(G)\setminus (V_\delta(G) \cup V_{\Delta}(G))$, and $g(G)=\Delta(G)-\delta(G)$.
We will reduce Theorem~\ref{thm:expander} into the statement below.

\begin{thm}\label{thm:thm-reduced}
Let $n\in \mathbb{N}$ and suppose that  $0 < \frac{1}{n}  \ll  \eta  \ll  \nu \ll \tau  \ll   \alpha  \ll 1$.  Let
	$G$ be  a   robust $(\nu, \tau)$-expander on $n$ vertices with $\delta(G)\geq \alpha n$ and
	$U(G) =\{v\in V(G) \mid \Delta(G)-d_G(v) \ge \eta   n\}$.
	Then $\la(G)\leq \lceil\frac{\Delta(G)+1}{2}\rceil$ if $G$ satisfies one of
	the following three conditions.
	\begin{enumerate}[(1)]
		\item   $|U(G)| \ge \eta n$ or ($g(G)  \ge  \eta n$ and $|V_\delta(G) \cup V_{\delta+1}(G)|  \ge  \eta n$);
		\item ($|U(G)| < \eta n $, $g(G) \le 2$, and $|W(G)| \ge 2$) or ($1<g(G)<\eta n$ and $|V_\delta(G) \cup V_{\delta+1}(G)| \ge \eta n$);
		\item $|U(G)| < \eta n $  and $|W(G)| \le 1$.
	\end{enumerate}
\end{thm}

\begin{proof}[Proof of Theorem~{\ref{thm:expander}} Assuming Theorem~\ref{thm:thm-reduced}]
	
 Let $\tau^*=\tau(\alpha)$ be  defined as in Theorem~\ref{thm:expander}.
We choose an additional parameter   $\nu'$   such that  $$0<1/n \le 1/n_0  \ll \eta \ll \nu'\ll  \nu  \le \tau  \ll  \tau^* \ll  \alpha <1.$$

 If $G$ meets any of the conditions of Theorem~\ref{thm:thm-reduced},
 then we get $\la(G)\leq \lceil\frac{\Delta(G)+1}{2}\rceil$. Thus we
suppose that  $|U(G)| < \eta n$,   $g(G) \ge 3$,   $|W(G)| \ge 2$, and  $|V_\delta(G) \cup V_{\delta+1}(G)|<\eta n$.  We first show that $|U(G) \cup V_\delta(G) \cup V_{\delta+1}(G)|<\eta n$.

\begin{claim}\label{claim:size-U0-V0-Z0}
We have $|U(G) \cup V_\delta(G) \cup V_{\delta+1}(G)|<\eta n$.
\end{claim}
\proof
We have $|U(G)|<\eta n$ and  $|V_\delta(G) \cup V_{\delta+1}(G)|<\eta n$ by the assumption on $G$.
As $V_\delta(G) \cup V_{\delta+1}(G)\subseteq U(G)$ if $g(G)\ge \eta n+1$, and $U(G)\subseteq V_\delta(G) \cup V_{\delta+1}(G)$ if $g(G)< \eta n+1$, we have
$|U(G) \cup V_\delta(G) \cup V_{\delta+1}(G)|<\eta n$ if $g(G) \ge \eta n$.
If $g(G) < \eta n$, then we have $U(G)=\emptyset$  and so  $|U(G) \cup V_\delta(G) \cup V_{\delta+1}(G)| =|V_\delta(G) \cup V_{\delta+1}(G)|<\eta n$.
\qed

  We will remove
edge-disjoint paths that  each contain all vertices of $V_\Delta(G)$ as internal vertices and reduce $G$
to a graph $G'$ such that $\delta(G')=\delta(G)$ and $G'$ satisfies one of the three conditions of Theorem~\ref{thm:thm-reduced}.
Our goal is to apply Theorem~\ref{thm:hamilton2}.  Thus, we start with identifying edges  to be included in $F_i$ as
needed in the theorem.
  Let
  \begin{eqnarray*}
  W_0&=&W(G), \quad U_0=U(G), \quad  V_0=V_\delta(G), \\
  Z_0&=&V_{\delta+1}(G),\quad  G_0^*=G-(U_0\cup V_0 \cup Z_0),  \\
   g_0&=&g(G),  \quad \text{and} \quad d_0(v)=d_G(v)  \quad  \text{for each $v\in V(G)$.}
  \end{eqnarray*}
We perform the following algorithm.
For each  $i \ge 1$,
if $g_{i-1}  \ge 3$,  $|W_{i-1}| \ge 2$,  and $|U_{i-1} \cup V_{i-1} \cup Z_{i-1}|<\eta n$,
then we let
\begin{enumerate}[(a)]
	\item $x_i, y_i \in W_{i-1}$ be distinct;
	\item $d_i(v) = d_{i-1}(v)-2$ for each $v\in V(G_{i-1}^*)\setminus (V_{i-1} \cup Z_{i-1})$ with  $v\not\in\{x_i, y_i\}$,
	$d_i(v)  =d_{i-1}(v)-1$ for $v\in \{x_i,y_i\}$, and   $d_i(v)=d_G(v)$ for all other vertices $v$;
	\item $g_i=g_{i-1}-2$ and $W_i=\{v\in V(G) \mid   \delta(G)<d_i(v) < \Delta(G)-2i\}$;
	\item $U_i=\{v\in V(G) \mid   \Delta(G)-2i-d_i(v) \ge \eta  n\}$;
	\item $V_i=\{v\in V(G) \mid   d_i(v)  =\delta(G)\}$ and $Z_i=\{v\in V(G) \mid   d_i(v)  =\delta(G)+1\}$;
	\item  $G_i^*=G-(U_i \cup V_i \cup Z_i)$.
\end{enumerate}

Note that  $U_i\subseteq U_0$ and $V_0\cup Z_0 \subseteq V_i\cup Z_i$ for each $i$.
Although the graph $G_0^*$  does not contain any vertex of $V_0\cup Z_0$ initially, it can contain vertices of $V_i\cup Z_i$
 in some later steps.  When $g_0\ge \eta n$, it is possible that $V(G_0^*) \subseteq V(G_i^*)$ for some $i$; also in the process,
 it is possible that  $V(G_i^*) \subseteq V(G_0^*)$ when $V(G_0^*) \cap (V_i\cup Z_i) \ne \emptyset$.

Since $g_i$ is strictly decreasing with $i$, the algorithm above  will end
after  less than $\lfloor g(G)/2\rfloor$ steps.
Suppose that upon the completion of the algorithm, all the vertex pairs
we found are $x_1,y_1, \ldots, x_\ell, y_\ell$.
As the process cannot proceed any further,   we have  $g_\ell \le 2$ or  $|W_\ell| \le 1$ or $|U_\ell \cup V_\ell \cup Z_\ell|\ge \eta n$.
The graphs $G_i^*$  and $G_j^*$ might be distinct for distinct $i$ and $j$. In order to
apply Theorem~\ref{thm:hamilton2} on one single host graph, we match vertices of each $G_i^*$ that are not in $G_0^*$  for $i\ge 1$
to some vertices of $G_0^*$.

For each $i\in[1,\ell]$,  we select the following vertices from $G_0^*$ and also construct $F_i$ and $L_i$:
\begin{enumerate}[(i)]
	\item If $x_i\in V(G_0^*)$, let $x_i'=x_i$; otherwise, let $x_i'x_i\in E(G)$ with $x_i'\in V(G_0^*)\setminus (V_{i-1} \cup Z_{i-1})$ such that $x_i'x_i$ has not been used in previous steps (this is possible, since $|V(G_0^*)\setminus (V_{i-1} \cup Z_{i-1})| \ge (1-2\eta) n$ and $\delta(G) \ge \alpha n$);
	\item If $y_i\in V(G_0^*)$, let $y_i'=y_i$; otherwise, let $y_i'y_i\in E(G)$ with $y_i'\in V(G_0^*)\setminus (V_{i-1} \cup Z_{i-1})$ such that $y_i'y_i$ has not been used in previous steps and $y'_i \ne x_i'$;
	\item We match each vertex  $z$ of $V(G_i^*)\setminus V(G^*_0) $  to two distinct vertices $z_{i1}, z_{i2}$ of $V(G_0^*)\setminus (\{x_i',y_i'\}  \cup V_{i-1} \cup Z_{i-1})$   such that
	$z$ is adjacent in $G$ to both $z_{i1}, z_{i2}$,  $z_{i1}$ and $z_{i2}$ have not  been matched to any other vertex of $V(G_i^*)\setminus V(G^*_0) $,
	and that $zz_{i1}, zz_{i2}$ have not been used  in any previous steps (since  $V_0 \cup Z_0\subseteq V_i \cup Z_i$  and $U_i\subseteq U_0$ for each $i$, we know that $|V(G_i^*)\setminus V(G^*_{0})|\le |U_{0} \setminus U_i|<\eta n-1$. As  $\delta(G) \ge \alpha n$, vertices $z_{i1}, z_{i2}$ exist; also note that when $|V(G_i^*)\setminus V(G^*_{0})|>0$  happens, then we have $g_{0} \ge \eta n$);
	
	In the selection of vertices in (i)-(iii) above, we will make sure that each vertex $v\in V(G_0^*)$
	is used  in  at most $\sqrt{\eta}n$  steps. This is possible:
	for each $F_i$, we select at most $2+2(\eta n-1) =2\eta n$ distinct vertices of $V(G_0^*)\setminus (V_{i-1} \cup Z_{i-1})$;
 there are at most $3\sqrt{\eta} \ell \le 3\sqrt{\eta}(1-\alpha) n/2$ vertices of $V(G_0^*)\setminus (V_{i-1} \cup Z_{i-1})$
	that have been used  in at least  $\lfloor\sqrt{\eta} n \rfloor$ previous steps; and
	 	each vertex from   $\{x_i,y_i,z\}$
	 has in $G$ at least $(\alpha-2\eta) n$ distinct neighbors from $V(G_0^*)\setminus (V_{i-1} \cup Z_{i-1})$.
	 Since $(\alpha-2\eta) n -3\sqrt{\eta}(1-\alpha) n/2 >2\eta n$,
	we can find the desired vertices at Step $i$  from $V(G_0^*)\setminus (V_{i-1} \cup Z_{i-1})$  that are used
	at most $\lfloor\sqrt{\eta} n \rfloor -1$ times in the previous steps.
 	
		\item Let $F_i$  be the linear forest consisting of edges $x_i'y_i'$ and $z_{i1}z_{i2}$ for each $z\in V(G_i^*)\setminus V(G^*_0) $.   Since   $|V(G_i^*)\setminus V(G^*_{0})|\le |U_{0} \setminus U_i|<\eta n-1$,
		$e(F_i)<\eta n$  for each $i\in [1,\ell]$. Furthermore, by the
		comment on the selection of vertices before (iv), we know that
		 every vertex of $G_0^*$ is used by at most $\sqrt{\eta} n$ distinct $F_i$'s.
		\item For each $i \in [1,\ell]$, let $v_{i1},v_{i2} \in V(G_0^*)\setminus (V(F_i) \cup V_{i-1} \cup Z_{i-1})$ be distinct and such that, for any $v \in V(G_0^*)$, there exist at most two $(i,j) \in [1,\ell]\times[1,2]$ such that $v = v_{ij}$. For each $i \in [1,\ell]$, denote by $P_{i,1},\dots,P_{i,f_i}$ the components of $F_i$ and, for $j \in [1, f_i]$, denote by $u_{ij}$ and $w_{ij}$ the starting and ending vertices  of $P_{i,j}$.
		For each $i \in [1,\ell]$, let   $L_i$ be the union of
		\[
	\{ v_{i1}u_{i1}P_{i,1}w_{i1}u_{i2}P_{i,2}w_{i2}u_{i3}\ldots w_{i,f_i}v_{i2},v_{i2}v_{i1}\},\]
	and  $
	 \left((V_{i-1} \cup Z_{i-1})  \cap V(G_0^*)\right)\setminus \{x_i,y_i\}$ (those vertices
	 are  isolated vertices in $L_i$, and their existence implies $g_{i-1} \le \eta n+1$).

		Denote $L = \bigcup_{i \in [1,\ell]} L_i$. Note that $(L_1,F_1),\dots,(L_\ell,F_\ell)$ are layouts such that, for each $i \in [1,\ell]$, $V(L_i)\subseteq V(G_0^*)$, $|V(L_i)| \le 2e(F_i)+2+|V_{i-1} \cup Z_{i-1}| \le 3 \eta n$ and $e(L_i) \le 2e(F_i)+3< 3 \eta  n$. Moreover, for each $v \in V(G_0^*)$, $d_L(v) \le 2\sqrt{\eta} n+2 $ and there exist at most $ \sqrt{\eta}  n +2+\eta n+1<2 \sqrt{\eta} n$ indices $i \in [1,\ell]$ such that $v \in V(L_i)$ (when $(V_{i}\cup Z_{i}) \cap V(G_0^*) \ne \emptyset$, we have $g_i \le \eta n+1$. Thus a vertex in $V_{i}\cup Z_{i} $ can be used by at most $\eta n$ other $L_j$'s for $j \ge i$).
\end{enumerate}

Let $\alpha^*= \Delta(G)/n$.   As $d_G(v) \ge (\alpha^*-\eta)n$ for any $v\in V(G_0^*)$,
  by Claim~\ref{claim:size-U0-V0-Z0},  $G_0^*$  is
  $(\alpha^*, 2\eta)$-almost regular.
Furthermore, $G_0^*$  is  a  robust $(\nu-\eta, 2\tau)$-expander by Lemma~\ref{lem:stability-expansion}.
 Applying Theorem~\ref{thm:hamilton2} with
 $G_0^*$, $\alpha^*$,  $\nu-2\eta$, $2\tau$, and $\alpha$ playing the roles of $G,\alpha,\nu,\tau$, and $\eta^*$,
 and with  $\ve^*=2\eta$ and $\ve= \eta^{1/7}$,
 we find  edge-disjoint  subgraphs $C_1, \ldots, C_\ell$  of  $G_0^* \cup (\bigcup_{i=1}^\ell F_i)$ such that  $C_i$ is a
 spanning configuration of shape $(L_i,F_i)$ and
 $E(F_i)\subseteq E(C_i)$.

 For each  $i\in [1,\ell]$,  for $x_i'y_i'\in E(F_i)\cap E(C_i)$, we delete $x_i'y_i'$ from $C_i$, and  add the edges  $x_i'x_i$ (if $x_i\ne x_i'$) and $y_i'y_i$ (if $y_i\ne y_i'$) to the
 resulting path; for any other edge $z_{i1}z_{i2}$ of $F_i$, we let $z\in V(G_i^*)\setminus V(G^*_0) $ such that $z$
 was matched to $z_{i1}, z_{i2}$ in step (iii) above, then we  replace $z_{i1}z_{i2}$  on  $C_i$  by the path
 $z_{i1}zz_{i2}$.  Denote by $P_i$ the resulting path. Then $P_i$ is a Hamilton $(x_i,y_i)$-path of $G[V(G_i^*) \cup \{x_i,y_i\}]$.
  Let $G'=G-\bigcup_{i=1}^{\ell}E(P_i)$.    As $G^*_0-\bigcup_{i=1}^\ell E(C_i)$ is a robust $(\nu'/2,8\tau)$-expander by Theorem~\ref{thm:hamilton2},
  $G'$ is a robust $(\nu'/2-\eta,16\tau)$-expander by Lemma~\ref{lem:stability-expansion}.  As $\tau \ll \tau^*$, $G'$ is also a robust $(\nu'/2-\eta,\tau^*)$-expander.
Note that   $g(G')=g_\ell$, $W(G')=W_\ell$,  $U(G') =U_\ell,   V_\delta(G') =V_\ell$,  $V_{\delta+1}(G')=Z_\ell$,  $U(G')\subseteq U(G)$, and $\delta(G')=\delta(G) \ge \alpha n$.
Since  $|U(G')| \le | U(G)|<\eta n$, and  ($g_\ell \le 2$ or  $|W_\ell| \le 1$ or $|U_\ell\cup V_\ell \cup Z_\ell| \ge \eta n$),  we have one of the following three
possibilities:
\begin{enumerate}[(A)]
	\item  $|W_\ell| \le 1$. In this case, Theorem~\ref{thm:thm-reduced}(3) holds.
	\item $|W_\ell|  \ge 2$ and $g_\ell \le 2$. In this case, Theorem~\ref{thm:thm-reduced}(2) holds.
	\item $|W_\ell|  \ge 2$, $g_\ell  \ge 3$, and $|U_\ell\cup V_\ell \cup Z_\ell| \ge \eta n$.
	\begin{itemize}
		\item If $g_\ell \ge \eta n+1$, then we have $V_\ell \cup Z_\ell \subseteq U_\ell$ and so  $|U_\ell|=|U_\ell\cup V_\ell \cup Z_\ell|  \ge \eta n$. This gives a contradiction to
		$|U_\ell| \le | U(G)|<\eta n$.
		\item If $ \eta n\le g_\ell <\eta n+1$, then we have $U_\ell =V_\ell$ and so  $|V_\ell\cup Z_\ell|=|U_\ell\cup V_\ell \cup Z_\ell|  \ge \eta n$. In this case, Theorem~\ref{thm:thm-reduced}(1) holds.
		\item If $g_\ell <\eta n$, then $U_\ell=\emptyset$ and so  $|V_\ell\cup Z_\ell|=|U_\ell\cup V_\ell \cup Z_\ell|  \ge \eta n$. Since also $g_{\ell}\geq 3$, in this case, Theorem~\ref{thm:thm-reduced}(2) holds.
	\end{itemize}
\end{enumerate}
Applying  Theorem~\ref{thm:thm-reduced},
we get  $\la(G') \le \lceil \frac{1}{2}(\Delta(G')+1 )\rceil $.
Thus  $\la(G) \le \ell+\lceil \frac{1}{2}(\Delta(G')+1 )\rceil  =\lceil \frac{1}{2}(\Delta(G)+1 )\rceil $.
\qed

The following lemma will be used in the proof of Theorem~\ref{thm:thm-reduced} when the second condition is met.
	 For a positive integer $d$,
	the \emph{$d$-deficiency} of a vertex $v$ in $G$ is $\df_{G}(v,d)=\max\{d-d_G(v), 0\}$.

	\begin{lem}\label{lem:reduce-to-regular}
	Let $n\in \mathbb{N}$ and suppose that $0<1/n \ll \nu  \le  \tau \ll  \alpha <1$.  Suppose that
	$G$ is a   robust $(\nu, \tau)$-expander on $n$ vertices with $\delta(G)\geq \alpha n$, and that
		for some positive integer $d$,    $\df_G(v_1,d) \ge \ldots \ge \df_G(v_n, d)$, and
		$\df_G(v_1,d) \le  \nu^6 n+1$.  If   $\sum_{i=1}^n\df_G(v_i, d)$ is even and  $\df_G(v_1,d) \le \sum_{i=2}^n\df_G(v_i,d)$, then there exists some integer $\ell \le 3 \nu^2 n$ such that after removing
		$\ell$ edge-disjoint linear forests from $G$, we can get a graph $G^*$
		such that $d_{G^*}(v_i)=d_G(v_i)-2\ell+\df_G(v_i,d)$ for each $i\in [1,n]$, and $G^*$ is a robust
		$(\nu/2, \tau)$-expander with $\delta(G^*) \ge (\alpha -6 \nu^2 )n$.
	\end{lem}
	
	\proof By Theorem~\ref{thm:degree-sequence-multigraph},  there exists a multigraph $H$ on $V(G)$ such that $d_{H}(x_i)=\df_G(v_i,d)$ for each $i\in [1,n]$. The multigraph $H$ will aid us to find a desired subgraph $G^*$ of $G$. Note that $\Delta(H)=\df_G(v_1,d)=d-d_G(v_1)\leq \nu^6 n+1$.
	
	By Vizing's Theorem on chromatic index, we have $\chi'(H)\leq \Delta(H)+\mu(H)\leq 2\Delta(H)\leq 2\nu^6n+2$, where $\mu(H)$
	is the maximum number of edges joining two vertices in $H$. By further partitioning the edges  in each color class
	of a $(2\nu^6n+2)$-edge coloring of $H$ into $\nu^{-4}$ matchings,
	we can greedily partition the edges of $H$ into
	$$
	\ell \le \frac{2\nu^6n+2}{\nu ^{4}}  <3 \nu^2 n
	$$
	matchings $M_1,M_2,\ldots,M_\ell$ each of size at most
	\begin{eqnarray*}
		\frac{\frac{n}{2}}{\nu^{-4}} &=&\nu^4n/2.
	\end{eqnarray*}

	Now we take out linear forests from $G$ by applying Lemma~\ref{lem:hamiltonicity-of-expander} with $M_1,M_2,\ldots,M_\ell$ iteratively.
	More precisely, we define spanning subgraphs $G_0,G_1,\ldots, G_\ell$ of $G$ and edge-disjoint linear forests $F_1,F_2,\ldots,F_\ell$ such that
	
	\begin{enumerate}[(a)]
		\item $G_0:=G$ and $G_i:=G_{i-1}- E(F_i)$ for $i\geq 1$.
		\item $F_i$ is a spanning linear forest in $G_{i-1}$ whose leaves are precisely the vertices in $V(M_i)$.
	\end{enumerate}
	
	Let $G_0=G$ and suppose that for some $i\in [1,\ell]$, we have already defined $G_0,G_1,\ldots,G_{i-1}$ and $F_1,\ldots,F_{i-1}$. Since $\Delta(F_1\cup\ldots \cup F_{i-1})\leq 2(i-1)\leq 2(\ell-1)\leq 6\nu^2n$, it follows that  $G_{i-1}$ is still
	a robust $(\nu-6\nu^2, \tau)$-expander and so is a robust $(2^{-1/4} \nu, \tau)$-expander. Since  $M_i$ has size at most $\nu^4 n/2$,  we can apply Lemma~\ref{lem:hamiltonicity-of-expander} to $G_{i-1}$ and $M_i$   to obtain
	a spanning linear forest $F_i$ in $G_{i-1}$ whose leaves are precisely the vertices in $V(M_i)$.
	Let $G_i=G_{i-1}-E(F_i)$.
	
	We claim that $G_\ell$ satisfies the degree constraint for Lemma~\ref{lem:reduce-to-regular}. Consider any vertex $u\in V(G_\ell)$. For every $i\in [1,\ell]$, $d_{F_i}(u)=1$ if $u$ is an endvertex of some edge of $M_i$ and $d_{F_i}(u)=2$ otherwise. Since $M_1,M_2,\ldots,M_\ell$ partition $E(H)$, we know that $\sum\limits_{i=1}^{\ell}d_{F_i}(u)=2\ell-d_{H}(u)=2\ell-\df_G(u,d)$. Thus $$d_{G_{\ell}}(u)=d_{G}(u)-\sum\limits_{i=1}^{\ell}d_{F_i}(u)=d_{G}(u)-2\ell+\df_G(u,d).$$
	Furthermore $\delta(G_\ell) \ge (\alpha -6 \nu^2 )n$.
	Thus letting $G^*=G_\ell$ gives us the desired subgraph.
	\qed

	\proof[Proof of Theorem~\ref{thm:thm-reduced}]
Let 	$\tau^*= \max\{\tau(\alpha/2), \tau(\alpha), \tau(\alpha-6\nu^2) \}$, where $\tau(\cdot)$  is  defined in Theorem~\ref{thm2.13}.
We choose  an additional parameter    $\nu'$  such that  $$0<1/n \le 1/n_0  \ll \eta \ll \nu'\ll  \nu  \le \tau  \ll    \tau^* \ll  \alpha <1,$$
where
$n_0 \ge \max\{N_0(\alpha/2,   \eta^2, \tau^*), N_0(\alpha-6\nu^2,   \nu'/3, \tau^*), N_0(\alpha,   \nu/2, \tau^*) \}$
	satisfying  $0< 1/n_0   \ll \eta$
	and $N_0(\cdot)$ is  defined in Theorem~\ref{thm2.13}.

	Let  $G$ be a robust $(\nu,\tau)$-expander   on $n\ge n_0$ vertices with $\delta(G) \ge \alpha n$, and
	 let
	 $$
	 \Delta=\Delta(G), \quad \delta=\delta(G),  \quad V_\delta =V_\delta(G),  \quad V_{\delta+1}=V_{\delta+1}(G), \quad W=W(G), \quad g=g(G).
	 $$
Additionally, define  $$U =\{v\in V(G)\mid \Delta(G)-d_G(v) \ge \eta n\}.$$
	
 For any two vertices of degree less than $\Delta$, if they are not adjacent in
	 $G$, we add an edge joining them to $G$. By iterating this process, we thus assume that
	 all vertices with degrees less than $\Delta$ are mutually adjacent in $G$. Thus,
	by Lemma~\ref{VDelta}, we have  $|V_\Delta| >\frac{n}{2}$.
	 Since $W\cup V_\delta$ is  a clique in $G$,  we have $\delta \ge   |W\cup V_\delta|-1$.
	 Thus
	 \begin{equation}\label{eqn:V_D-size}
	 |V_\Delta|  >  \max\{n/2, n-\delta-2\}.
	 \end{equation}
	
	   If $G$ is regular, then we are done by Theorem~\ref{cor:regular-expander}. Thus we assume that $G$ is not regular and so $g\ge 1$.
	 This together with the fact that $W\cup V_\delta$ is a clique in $G$ imply
	 \begin{equation}\label{eqn:V_D-size2}
	 	\Delta \le n-2.
	 \end{equation}

	We  let  $V(G)=\{v_1, \ldots, v_n\}$, and
	 we assume, without loss of generality, that $$d_G(v_1) \le \ldots \le d_G(v_n).$$
	 We  proceed with  the proof according to the three given cases.
	
	 {\bf \noindent Case 1: $|U|\geq \eta n$ or  ($g \ge  \eta n$ and $|V_\delta\cup V_{\delta+1}|  \ge  \eta n$).}

	 If $\Delta$ is odd, we delete a perfect matching $M$ from $G$ if $n$ is even, and we delete a matching covering all
	 vertices of $V(G)\setminus \{v_1\}$ if $n$ is odd, where recall that we assumed  $d_G(v_1)=\delta$ and $M$ exists as $G$ is Hamiltonian by Lemma~\ref{lem:hamiltonicity-of-expander} with $k=1$. Thus in the following,  we let $\Delta^*=\Delta$ and $G^*=G$  if $\Delta$
	 is even,  and $\Delta^*=\Delta-1$ and $G^*=G-M$  if $\Delta$ is odd. Note that $\sum_{i=1}^n \df_{G^*}(v_i,\Delta^*) \ge (\sum_{i=1}^n \df_G(v_i, \Delta)) -1$,  here we have this extra ``$-1$''
	 in the bound
	 since when both $\Delta$ and $n$ are odd, we have $\Delta^*-d_{G^*}(v_1) =\Delta-d_G(v_1)-1$.
	  We will add vertices to $G^*$ and construct a simple $\Delta^*$-regular graph $H$ in three steps such that $G^*\subseteq H$, and $H$ is a robust $(\eta^2,2\tau)$-expander.

	 Let $\df(G^*)=\sum_{i=1}^n \df_{G^*}(v_i,\Delta^*)$. Then by the condition of this case and the way of reducing $G$ to $G^*$,
	 we have  $\df(G^*)\ge  \eta n(\eta n-1)-1 \ge \eta^2 n^2 /2  \ge \Delta^*+5$.
	 We add a set $X$ of new vertices to $G^*$ such that $|X| \ge \frac{\Delta^*+4}{2}$ is the smallest integer satisfying the following properties:
	 \begin{enumerate}
	 	\item[(i)] $|X|\equiv n\ (\mbox{mod}\ 2)$;
	 	\item[(ii)] $|X|\geq \Delta^*-\lfloor\frac{{\rm def}(G^*)}{|X|}\rfloor+4$.
	 \end{enumerate}
	
	 Note that when $|X|=\Delta^*+4$ or $|X|=\Delta^*+5$, the conditions above are satisfied. Thus a smallest integer $|X|$ satisfying the conditions above
	 exists and $|X|\leq \Delta^*+5\leq n+4$.
	
	Let $\delta^*=\delta(G^*)$.
	 We claim  that   $(\Delta^*+4)^2 -4\df(G^*)> (\Delta^* - 2\delta^*)^2$. To see this,
	 we have
	 $\df(G^*) \le (\Delta^*-\delta^*) \times |W\cup V_\delta| < (\Delta^*-\delta^*) (\delta^*+2)$ by $|V_\Delta|  >  \max\{n/2, n-\delta-2\}$ from~\eqref{eqn:V_D-size} (the set of maximum degree vertices of $G$ and $G^*$ stay the same) and so
	 \begin{eqnarray}
	(\Delta^*+4)^2 -4\df(G^*)  &\ge &  (\Delta^*+4)^2 -4\Delta ^* \delta^*+4 (\delta^*)^2  -8(\Delta^*-\delta^*)\nonumber  \\
	&= & (\Delta^* - 2\delta^*)^2+8\delta^*+16 >0. \label{eqn:X-size}
	 \end{eqnarray}
Now from (ii) that $|X|\geq \Delta^*-\lfloor\frac{{\rm def}(G^*)}{|X|}\rfloor+4$, we get
$|X|^2-(\Delta^*+4)|X| +\df(G^*) \ge 0$. As $|X| \ge \frac{\Delta^*+4}{2}$ and~\eqref{eqn:X-size}, solving
the quadratic inequality, we get
\begin{eqnarray}
|X| &\ge& \frac{\Delta^*+4 +\sqrt{(\Delta^*+4)^2 -4\df(G^*)}}{2}  \nonumber \\
	 &>& \frac{\Delta^*+4 +\sqrt{(\Delta^* - 2\delta^*)^2}}{2} \nonumber  \\
	 	&>& \Delta^*-\delta^*, \label{eqn:X-size2}
\end{eqnarray}
where the last inequality above holds regardless  whether  $\Delta^*-2\delta^*\ge 0$.

Let $d=\Delta^*-(\lceil\frac{{\rm def}(G^*)}{|X|}\rceil-1)$.
We also claim that $d-1  \ge  |X|-6$. If $|X| <  \frac{\Delta^*}{2}+4$, then as $|X| \ge  \frac{\Delta^*+4}{2}$ and  $\df(G^*) < \frac{1}{4}(\Delta^*+4)^2$, we have $d-1=\Delta^*-\lceil\frac{{\df}(G^*)}{|X|}\rceil+1-1\geq\frac{\Delta^*}{2}-2>|X|-6$.
If $|X|  \ge   \frac{\Delta^*}{2}+4$,  then by the  minimality  requirement
on the size of $X$, we know that $|X|-2< \Delta^*-\lfloor\frac{{\rm def}(G^*)}{|X|}\rfloor+4 $ and so  $|X|-6\le \Delta^*-\lceil\frac{{\df}(G^*)}{|X|}\rceil+1-1=d-1$.

	 Suppose that ${\rm def}(G^*)\equiv \ell\ (\mbox{mod}\ |X|)$, where $\ell\in[0, |X|-1]$. Let $d_1=\ldots= d_{\ell}=\Delta^*-\lceil\frac{{\rm def}(G^*)}{|X|}\rceil=d-1$, and $d_{\ell+1}=\ldots=d_{|X|}=\Delta^*-(\lceil\frac{{\rm def}(G^*)}{|X|}\rceil-1)=d$. Since $|X| \equiv n\ (\mbox{mod}\ 2)$, we know that
	 \begin{equation*}\sum\limits_{i=1}^{|X|}d_i=\Delta^*|X|-{\rm def}(G^*)\equiv \Delta^*n-{\rm def}(G^*)=2e(G^*)\equiv 0\ (\mbox{mod}\ 2).
	 \end{equation*}
	
	 Furthermore, we have
	 \begin{equation*}
	 	d=\Delta^*-\left(\left\lceil\frac{{\df}(G^*)}{|X|}\right\rceil-1 \right) \ge \Delta^*-\frac{(\Delta^*+4)^2/4}{(\Delta^{*}+4)/2}= \Delta^*-\frac{\Delta^*+4}{2}>2.
	 \end{equation*}
	 Let $X=\{x_1,\ldots,x_{|X|}\}$. By Lemma~\ref{graphic}, there is a graph $R$ on $X$ such that $d_{R}(x_i)=d_i$ for each $i\in [1,|X|]$.
	 We now add ${\rm{def}}(G^*)$ edges joining vertices from $V(G^*)\setminus V_{\Delta}$ to $X$. We list these ${\rm{def}}(G^*)$ edges as $e_1, e_2,\ldots, e_{{\rm{def}}(G^*)}$ and add them one by one such that
	 \begin{enumerate}
	 	\item[(a)] For each vertex $v\in V(G^*)\setminus V_{\Delta}$, the ${\rm def}_{G^*}(v)$ edges joining $v$ with the vertices of $X$ are listed consecutively in the ordering above.
	 	\item[(b)] For each $x_i\in X$, $x_i$ is incident with $\Delta^*-d_i$ edges for each  $i\in [1,|X|]$.
	 \end{enumerate}

	 The resulting multigraph is denoted by $H$. Then $H$ is $\Delta^*$-regular by the above construction. The multigraph $H$ is actually simple, since    $|X|>\Delta^*-\delta^*$ by~\eqref{eqn:X-size2}.
	 Let $|V(H)|=n+|X|=n'$.  By  $|X| \le n+4$ and $\Delta^*\ge \delta+\eta n-1>\alpha n+2$, it follows that $\Delta^* >\alpha n'/2$.
	 In the rest of the proof, we mainly show that $H$ is a robust $(\eta^2,2\tau)$-expander.
	 Let $S\subseteq V(H)$ with  $2\tau n'\leq |S|\leq (1-2\tau)n'$. We show that $|RN_{H}(S)|\geq |S|+\eta^2 n'$.
	
	 Consider first that $|S\cap V(G^*)|\geq \tau n$.
	 Since $G^*$ is  a  robust $(\nu/2,\tau)$-expander implies it is a  $(5\eta^2,\tau)$-expander, we have $|RN_{G^*}(S)|\geq |S\cap V(G^*)|+5\eta^2 n$.
	 If $|S\cap X| \le  \eta^2 n'+6$, then
	 \begin{equation*}\begin{aligned}|RN_{H}(S)|&\geq|RN_{G^*}(S)|\geq |S\cap V(G^*)|+|S\cap X|+2.5\eta^2 n\\&\geq |S|+\eta^2 n',
	 	\end{aligned}
	 \end{equation*}
	 as  $|X| \le n+4$ and so $n'\le 2n+4$.
	 Thus we
	 assume that $|S\cap X|>\eta^2 n'+6$. Then  as $\delta(R)=d-1$ and    $d-1 \ge  |X|-6$,
	 we get $X\subseteq RN_H(S)$.
	 Therefore,
	 \begin{equation*}
	 	|RN_{H}(S)|\geq |RN_{G^*}(S)|+|X\cap S|\geq|S\cap V(G^*)|+5\eta^2 n+|X\cap S|\geq |S|+\eta^2 n'.
	 \end{equation*}
	
	 Consider then that $|S\cap V(G^*)|< \tau n$. As $|S|\geq 2\tau n'$, we have that $|S\cap X|> \tau  n>\eta^2 n'+6$. Again, as $\delta(R)=d-1  \ge  |X|-6$, we have $X\subseteq RN_{H}(S)$.
	
	 Assume for now  that $|S\cap V(G^*)|>\eta n/2$.  Then as $\delta(G^*)  \ge \alpha n-1$, we get
	 \begin{eqnarray*}
	 	|S\cap V(G^*) |	 (\alpha n-1) &\le&   \sum_{v\in S\cap V(G^*)} d_{G^*}(v)  \\
	 	&\le & |RN_{G^*}(S\cap V(G^*))| \times |S\cap V(G^*) |	 + \eta^2 n' \times n.
	 \end{eqnarray*}
	 This gives $ |RN_{G^*}(S\cap V(G^*))|   \ge (\alpha n-1)- \frac{\eta^2 n' \times n}{|S\cap V(G^*) |}  \ge (\alpha n-1) -5 \eta n > \tau n+\eta^2 n'$.
	 Thus
	 $$|RN_{H}(S)|\geq |X|+|RN_{G^*}(S\cap V(G^*))| \geq |X\cap S|+ \tau n+\eta^2 n'\geq |S|+\eta^2 n'.$$

	 Assume then that $|S\cap V(G^*)| <\eta n/2$.
	 If $|X\backslash S|\geq \eta^2 n'+\eta n/2$, then
	 $$|RN_{H}(S)|\geq |X|=|S|-|S\cap V(G^*)|+|X\setminus S|\geq |S|-\eta n/2+\eta^2 n'+\eta n/2=|S|+\eta^2 n'.$$
	 Therefore, we have $|X\backslash S|< \eta^2 n'+\eta n/2$. Let $U^*=U$ if $|U| \ge \eta n$ and $U^*=V_\delta \cup V_{\delta+1}$ otherwise.
Since every vertex from $U^*$ has in $H$ at least $\eta n-1$ neighbors   from $X$,
	 we then know that   $U^{*}\subseteq RN_H(S)$ and so
	 $$|RN_{H}(S)|\geq |X|+|U^*|\geq |X\cap S|+|V(G^*)\cap S|+\eta n/2\geq |S|+\eta^2 n'.$$

	 Thus $H$ is a $\Delta^*$-regular  robust $(\eta^2,2\tau)$-expander, where $\Delta^* > \alpha n'/2$ is even.
	 As $\tau  \ll \tau^*$, $H$ is a robust $(\eta^2, 2\tau)$-expander implies that $H$ is  a robust $(\eta^2, \tau^*)$-expander.
	 As $n$ is taken to be at least  $N_0(\alpha/2,   \eta^2, \tau^*)$  as defined in Theorem~\ref{thm2.13}, the theorem  implies that  $H$ has a Hamilton decomposition into $ \Delta^*/2 $ Hamilton  cycles.
	 By deleting the edges incident with vertices of $X$ from the decomposition, we obtain a decomposition of $G^*$ into $ \Delta^*/2 $ linear forests.  If $\Delta=\Delta^*$, we get $\la(G) \le \Delta/2$.
	 Otherwise, we get $\la(G) \le 1+(\Delta-1)/2 =(\Delta+1)/2$.

	 {\bf \noindent Case 2:  ($|U| < \eta n $, $g \le 2$,  and $|W| \ge 2$) or ($1<g<\eta n$ and $|V_\delta \cup V_{\delta+1}| \ge \eta n$).}

	 Since  $d_G(v_1) \le \ldots \le  d_G(v_n)$,  we have
	 $\df_G(v_1, \Delta) \ge \ldots  \ge \df_G(v_n, \Delta)$.
	 If $\sum_{i=1}^n\df_G(v_i, \Delta)$ is even and  $\df_G(v_1, \Delta) \le \sum_{i=2}^n\df_G(v_i, \Delta)$,
	 applying Lemma~\ref{lem:reduce-to-regular} with $d=\Delta$, we can reduce $G$ into a regular graph  $G^*$ with degree $\Delta-2\ell$
	 by taking out $\ell\le 3\nu^2n$ edge-disjoint linear forests.  Then $G^*$ is still a  robust $(\nu/2,\tau)$-expander  and so a  robust $(\nu/2,\tau^*)$-expander  with $\delta(G^*) \ge (\alpha-6\nu^2)n$.
	 By Theorem~\ref{cor:regular-expander}, we have
	 $\la(G^*) \le  \lceil (\Delta-2\ell+1)/2 \rceil $. Thus  $\la(G) \le  \ell+\lceil (\Delta-2\ell+1)/2 \rceil =\lceil (\Delta+1)/2 \rceil $.

	 Thus we assume that $\sum_{i=1}^n\df_G(v_i, \Delta)$ is odd or $\df_G(v_1, \Delta) > \sum_{i=2}^n\df_G(v_i, \Delta)$.
	If $|V_\delta \cup V_{\delta+1}|  \ge \eta n$,  as $g>1$, then we always have
	 $\df_G(v_1, \Delta)  \le   \sum_{i=2}^n\df_G(v_i, \Delta)$.  If $|V_\delta \cup V_{\delta+1}| <\eta n$, $g\le 2$, and $|W| \ge 2$,
	 then as $\df_G(v_1, \Delta)\leq2$,  and $ \sum_{i=2}^n\df_G(v_i, \Delta) \ge \sum_{w\in W}\df_G(w, \Delta) \ge 2$, it follows that $\df_G(v_1, \Delta) \le \sum_{i=2}^n\df_G(v_i, \Delta)$.
	 Therefore,   $\sum_{i=1}^n\df_G(v_i, \Delta)$ is odd and $\df_G(v_1, \Delta)  \le  \sum_{i=2}^n\df_G(v_i, \Delta)$.
	 Since  $\sum_{i=1}^n\df_G(v_i, \Delta)=n\Delta -2e(G)$, it follows that both $n$ and $\Delta$ are odd.
	 Since $\sum_{i=1}^n\df_G(v_i, \Delta)$ is odd  and $\df_G(v_1, \Delta) \le  \sum_{i=2}^n\df_G(v_i, \Delta)$, we know
	 $\df_G(v_3, \Delta) >0$.  Let  $P$ be a Hamilton $(v_1,v_n)$-path in $G$ by Lemma~\ref{lem:hamiltonicity-of-expander}.
	 Let $G_1=G-E(P)$. Then we have  $d_{G_1}(v)=d_G(v)-2$ for $v\in V(G_1)$ with $v\not\in \{v_1, v_n\}$, $d_{G_1}(v_1)=\delta-1$ and $d_{G_1}(v_n)=\Delta-1$. Thus  $\Delta(G_1)=\Delta-1$ and  $ \delta-2\le \delta(G_1) \le \delta-1$.
	
	 For every vertex $v\in V_\Delta$, by the definition, we have $\df_{G_1}(v, \Delta-2)=\df_G(v, \Delta)=0$;
	 for every vertex $v\in V(G)\setminus V_\Delta$ with $v\ne v_1$, we have $\df_{G_1}(v, \Delta-2)=\df_G(v, \Delta)$,
	 and $\df_{G_1}(v_1, \Delta-2)=\df_G(v_1, \Delta)-1=\Delta-\delta-1=g-1<\eta n-1$. Therefore, we get
	 $\sum_{i=1}^{n} \df_{G_1}(v_i, \Delta-2)=\sum_{i=1}^n\df_G(v_i, \Delta) -1$ and so $\sum_{i=1}^{n} \df_{G}(v_i, \Delta-2)$
	 is even.  We  next show that the largest value among  $\df_{G_1}(v_1, \Delta-2), \ldots, \df_{G_1}(v_n, \Delta-2)$ is at most
	 the sum of the rest values.
	 If $\df_{G_1}(v_1, \Delta-2)\ge\df_{G_1}(v_2, \Delta-2)$,  then $\df_{G_1}(v_1, \Delta-2)$ is a largest value
	 among $\df_{G_1}(v_1, \Delta-2), \ldots, \df_{G_1}(v_n, \Delta-2)$. Then
	 we have $\df_{G_1}(v_1, \Delta-2) <\df_G(v_1, \Delta) \le  \sum_{i=2}^n\df_G(v_i, \Delta)=\sum_{i=2}^{n} \df_{G_1}(v_i, \Delta-2)$.
	 Otherwise, we have  $\df_{G_1}(v_1, \Delta-2)<\df_{G_1}(v_2, \Delta-2)$.
	 Then we have $\df_{G_1}(v_2, \Delta-2) =\df_{G_1}(v_1, \Delta-2)+1$
	 as $\df_{G_1}(v_1, \Delta-2)+1=\df_G(v_1, \Delta) \ge \df_G(v_2, \Delta)=\df_{G_1}(v_2, \Delta-2)$.  In this case, $\df_{G_1}(v_2, \Delta-2)$ is a  largest value among
	 $\df_{G_1}(v_1, \Delta-2), \ldots, \df_{G_1}(v_n, \Delta-2)$.
	 Since $\df_{G_1}(v_3, \Delta-2)=\df_G(v_3, \Delta) >0$, we have
	 $\df_{G_1}(v_2, \Delta-2)\le \df_{G_1}(v_1, \Delta-2)+\sum_{i=3}^{n} \df_{G_1}(v_i, \Delta-2)$.
	 Furthermore, we have $\delta(G_1) \ge \delta(G)-1 \ge\alpha ^*n$ where $\alpha^* \ge 0.99\alpha $,  $\df_{G_1}(v_1, \Delta-2) \le \Delta-\delta +1 <\eta n+1$,  and  $\df_{G_1}(v_2, \Delta-2) <\eta n$.
	 Therefore, the graph $G_1$ and sequence of integers $\df_{G_1}(v_1, \Delta-2), \ldots, \df_{G_1}(v_n, \Delta-2)$
	 satisfy  the conditions  of  Lemma~\ref{lem:reduce-to-regular}.

	 Now applying Lemma~\ref{lem:reduce-to-regular} with  $G_1$ in the place of $G$
	 and with $d=\Delta-2$, we can reduce $G_1$ into a  graph  $G^*$
	 by taking out $\ell \le 3 \nu^2 n$ edge-disjoint linear forests such that for any $v\in V(G_1)$, we
	 have $d_{G^*}(v)=d_{G_1}(v)-2\ell+\df_{G_1}(v, \Delta-2)$.
	 For any $i\in [1,n-1]$, as $d_{G_1}(v_i) \le \Delta-2$, it follows that $d_{G_1}(v_i)+\df_{G_1}(v_i, \Delta-2)=\Delta-2$.
	 Thus all vertices except $v_n$ of $G^*$ have
	 degree $\Delta-2-2\ell$  and $v_n$ has degree  $\Delta-1-2\ell$, and $G^*$ is a robust $(\nu/2,\tau)$-expander.

	 Applying Lemma~\ref{lem:matching-in-complement}, $\overline{G^*}$  has a matching
	 excluding the vertex $v_n$ and covering at least  $n- \frac{n}{n-(\Delta-2-2\ell)}-3$
	 vertices of $G^*$.  Since $\ell \le  3 \nu^2 n$ and $\Delta \ge \alpha n$,
	 we have $(\Delta-2-2\ell)+1>\frac{n}{n-(\Delta-2-2\ell)}+2$, and so
	 $n- \frac{n}{n-(\Delta-2-2\ell)}-3>n-1-((\Delta-2-2\ell)+1)$.
	 Thus $\overline{G^*}$ has a matching $M$
	 covering exactly $n-1-((\Delta-2-2\ell)+1)$ vertices excluding the vertex $v_n$.
	 We let $G_2$ be obtained from $G^*+M$ by adding a new vertex $x$ and all the edges
	 joining $x$ and vertices from $V(G^*)\setminus (V(M) \cup \{v_n\})$.
	 The graph $G_2$ is $((\Delta-2-2\ell)+1)$-regular and is a  robust $(\nu/2,\tau)$-expander and so a robust $(\nu/2,\tau^*)$-expander.
	 Thus $G_2$  has a Hamilton decomposition by Theorem~\ref{thm2.13}.
	 Deleting all the edges of $G_2$ incident with $x$    and all the edges of $M$ from the decomposition gives
	 a linear forest decomposition of $G^*$ into $\lceil((\Delta-2-2\ell)+1)/2\rceil$ linear forests. Thus
	 $\la(G)  \le 1+\ell+ \lceil((\Delta-2-2\ell)+1)/2\rceil=\lceil(\Delta+1)/2\rceil$, as desired.

	  {\bf \noindent Case 3: $|U|<\eta n$ and $|W|\le 1$.}

	We may assume  that $|V_\delta|<\eta n$.  Indeed,
	if $g\ge \eta n$, then we have $V_\delta \subseteq U$ and so $|V_\delta|<\eta n$; and if
 $g<\eta n$ and  $|V_\delta| \ge \eta n$, then we are done by Case 2, thus
$|V_\delta|<\eta n$ again.

		 As $g=\Delta-\delta\leq n-2-\delta$ by Inequality~\eqref{eqn:V_D-size2}, we have
	 $|V_{\Delta}|  \ge g+1$ by Inequality~\eqref{eqn:V_D-size}.  If $|W|=1$, let $w\in W$
	 and $g_0=d_G(w)-\delta$.   We will remove from $G$ some edge-disjoint paths
	 and reduce it to a graph whose maximum degree and minimum degree differ by exactly one.
	 Recall that by our definition of $W$, we have $V_\Delta =V(G)\setminus (V_\delta \cup W)$.
	
	 We will select  $2\lceil g/2\rceil$ vertices  $x_1,y_1,\ldots,x_{\lceil g/2 \rceil},y_{\lceil g/2\rceil}$
	 and  take a Hamilton path  of $G-V_\delta$  or $G-V_\delta -W$ between  $x_i$ and $y_i$ for each $i\in [1,\lceil g/2\rceil]$, where
	 we will let $w$ be exactly one of those selected vertices  if $|W|=1$ and $g_0$ is odd.   Each pair $x_i,y_i$
	 will  be the two endvertices of a path connecting them.
	 We
	 describe below how to  select those vertices.
	
	Let $\alpha^*=\Delta/n$.
	 As $|V_\delta|<\eta n$, $G-V_\delta$ is  a robust $(\nu-\eta,2\tau)$-expander by Lemma~\ref{lem:stability-expansion}.
	 Thus $G-V_\delta$
	 is Hamilton-connected by Lemma~\ref{lem:hamiltonicity-of-expander}.
	 If   $|W|=1$ and $g_0$ is odd,  we let $x_1=w$ and $y_1$ be any vertex
	 from $V_\Delta$, and let $P_1$ be a Hamilton $(x_1,y_1)$-path of $G-V_\delta$.
	 Otherwise,  we let $x_1, y_1\in V_\Delta$ be any two distinct vertices, and let
	 $P_1$ be a Hamilton $(x_1,y_1)$-path of $G-V_\delta-W$.   Denote $G_0=G-V_\delta-W-E(P_1)$.
	
	 Suppose  $P_1=u_1u_2 \ldots u_p$, where $p=n-|V_\delta|$ if $|W|=1$ and  $g_0$ is odd and $p=n-|V_\delta|-|W|$ otherwise, $u_1=x_1$ and $u_p=y_1$.
	 Note that $p \ge g+1$ and  $G_0$ is    an  $(\alpha^*,2\eta)$-almost regular robust $(\nu-2\eta,2\tau)$-expander by Lemma~\ref{lem:stability-expansion}.
	 For each $i\in [2,\lceil g/2 \rceil]$, we let
	 $x_i=u_{2(i-1)}$ and $y_i=u_{2(i-1)+1}$.  Let $\ell = \lceil g/2 \rceil-1$.

	 If $g_0\le 1$, for each $i\in [2, \lceil g/2 \rceil]$, we let $F_i=x_iy_i$ be a linear forest.
	 For each $i \in [2,\ell+1]$, let $v_{i1},v_{i2} \in V(G_0)\setminus V(F_i)$ be distinct and such that, for any $v \in V(G_0^*)$, there exist at most two $(i,j) \in [2,\ell+1]\times[1,2]$ such that $v = v_{ij}$.
	 For each $i \in [2,\ell+1]$, let
	 \[
	 L_i=\{ v_{i1}x_iy_iv_{i2},v_{i2}v_{i1}\}.\]
	 	 Denote $L = \bigcup_{i \in [2,\ell+1]} L_i$. Note that $(L_2,F_2),\dots,(L_{\ell+1},F_{\ell+1})$ are layouts such that, for each $i \in [2,\ell+1]$, $V(L_i)\subseteq V(G_0)$, $|V(L_i)| =2e(F_i)+2<\eta n$ and $e(L_i) \le 2e(F_i)+3<  \eta  n$. Moreover, for each $v \in V(G_0)$, $d_L(v) \le 1+2 <\eta  n$ and there exist at most $ 1+2<\eta n$ indices $i \in [2,\ell+1]$ such that $v \in V(L_i)$.

	Applying Theorem~\ref{thm:hamilton2} with
	$G_0$, $\alpha^*$,  $\nu-2\eta$, $2\tau$, and $\alpha$ playing the roles of $G,\alpha,\nu,\tau$, and $\eta^*$,
	and with $\ve^* =2\eta$,  $\ve= \eta^{1/5}$,  and $\ell = \lceil g/2 \rceil-1 \le \frac{1}{2}(\alpha^*-\alpha) n$,
	 we find $\ell$ edge-disjoint Hamilton cycles $C_2,  \ldots, C_{\ell+1}$  of  $G_0+\{x_iy_i\mid i\in [2, \lceil g/2 \rceil]\}$
	 such that $E(F_i)\subseteq E(C_i)$ (as $L_i$ has no isolated vertices, a spanning configuration of shape $(L_i,F_i)$ is a Hamilton cycle of $G_0 \cup F_i$). Let $P_i=C_i-x_iy_i$. Then $P_i$ is a Hamilton $(x_i,y_i)$-path of $G_0-\bigcup_{j=2}^{i-1} E(P_j)$.
	  By Theorem~\ref{thm:hamilton2},  $G_0- \bigcup_{i=2}^{ \lceil g/2 \rceil} E(P_i)$
	 is a robust $(\nu'/2, 8\tau)$-expander.

	 If $g_0 \ge 2$, for each $i\in [2, \lceil (g_0+1)/2 \rceil]$,  and for the vertex $w\in W$,
	 we let $w_{i1},w_{i2}$ be two distinct vertices of $V_\Delta\setminus \{x_i,y_i\}$ such that
	 they are adjacent in $G$ to $w$, and the edges $ww_{i1},ww_{i2}$ have not been used so far.
	 We let $F_i $ be the  linear forest consisting of edges $x_iy_i$ and $w_{i1}w_{i2}$.  (This step is to reduce
	 the  degree of the vertex $w$ to $\delta$).   For each $i\in [\lceil (g_0+1)/2 \rceil+1, \lceil g/2 \rceil]$, we let $F_i=x_iy_i$ be a linear forest.
	
	  For each $i \in [2,\ell+1]$, let $v_{i1},v_{i2} \in V(G_0)\setminus V(F_i)$ be distinct and such that, for any $v \in V(G_0)$, there exist at most two $(i,j) \in [2,\ell+1]\times[1,2]$ such that $v = v_{ij}$.
	 For each $i\in [2, \lceil (g_0+1)/2 \rceil]$, let
	 \[
	 L_i=\{ v_{i1}x_iy_iw_{i1}w_{i2}v_{i2},v_{i2}v_{i1}\}.\]
	 For each $i\in [\lceil (g_0+1)/2 \rceil+1, \lceil g/2 \rceil]$, let
	 \[
	 L_i=\{ v_{i1}x_iy_iv_{i2},v_{i2}v_{i1}\}.\]
	 Denote $L = \bigcup_{i \in [2,\ell+1]} L_i$. Note that $(L_2,F_2),\dots,(L_{\ell+1},F_{\ell+1})$ are layouts such that, for each $i \in [2,\ell+1]$, $V(L_i)\subseteq V(G_0)$, $|V(L_i)| \le e(F_i)+4<\eta n$ and $e(L_i) \le 2e(F_i)+3<  \eta  n$. Moreover, for each $v \in V(G_0)$, $d_L(v) \le 1+2 <\eta  n$ and there exist at most $ 1+2<\eta n$ indices $i \in [1,\ell]$ such that $v \in V(L_i)$.

	 Applying Theorem~\ref{thm:hamilton2} with
	 $G_0$, $\alpha^*$,  $\nu-2\eta$, $2\tau$, and $\alpha$ playing the roles of $G,\alpha,\nu,\tau$, and $\eta^*$,
	 and with  $\ve^*=2\eta$,  $\ve= \eta^{1/5}$,  and $\ell = \lceil g/2 \rceil-1 \le \frac{1}{2}(\alpha^*-\alpha) n$,
	    we find edge-disjoint Hamilton cycles $C_2, \ldots, C_{\ell+1}$
	  of $G_0 \cup (\bigcup_{i=2}^{\ell+1} F_i)$ such that  $E(F_i)\subseteq E(C_i)$.  For each $i\in  [2, \lceil (g_0+1)/2 \rceil]$, we let $P_i$ be obtained from $C_i$ by deleting $x_iy_i$ and replacing $w_{i1}w_{i2}$ by  $w_{i1}ww_{i2}$;
	  and for each $i\in [\lceil (g_0+1)/2 \rceil+1, \lceil g/2 \rceil]$, we let $P_i$ be obtained
	  from $C_i$ by deleting $x_iy_i$. Then $P_i$ is
	 a Hamilton  $(x_i,y_i)$-path of  $G-V_\delta-E(P_1)-\bigcup_{j=2}^{i-1} E(P_j)$ if $i\in  [2, \lceil (g_0+1)/2 \rceil]$, and is a Hamilton  $(x_i,y_i)$-path
 of $G_0-\bigcup_{j=2}^{i-1} E(P_j)$ otherwise.
	  Furthermore,  by Theorem~\ref{thm:hamilton2},  $G_0- \bigcup_{i=2}^{ \lceil g/2 \rceil} E(P_i)$
	  is a robust $(\nu'/2, 8\tau)$-expander.

	 Let $G_1=G- \bigcup_{i=1}^{ \lceil g/2 \rceil} E(P_i)$ and $Z=\{x_1,y_1,\ldots,x_{\lceil g/2 \rceil},y_{\lceil g/2\rceil} \}$.
	  Since $|V_\delta\cup W|<\eta n+1$ and $G_0- \bigcup_{i=2}^{ \lceil g/2 \rceil} E(P_i)$
	  is a robust $(\nu'/2, 8\tau)$-expander,
	 it follows   by Lemma~\ref{lem:stability-expansion}  that $G_1$  is a robust $(\nu'/3, 16\tau)$-expander.

	 By the construction above,
	 vertices of $V_\delta$ were not involved in the  process. Thus we have $d_{G_1}(v)=\delta$
	 for each $v\in V_\delta$. Every vertex $v$ of $V_\Delta\setminus Z$ was used as an internal vertex for
	 each of the path removed, thus we have $d_{G_1}(v)=\Delta -2\lceil g/2\rceil$.
	 If $|W|=0$, then every vertex of $Z$ was used as an endvertex of exactly one of the paths removed,
	 and as an internal vertex of all other paths, thus we have $d_{G_1}(v)=\Delta -2\lceil g/2\rceil+1$
	 for each $v\in Z$.
	 If $|W|=1$, then the vertex $w$ was  used exactly $\lceil g_0/2 \rceil$ times,
	 where exactly $\lceil (g_0+1)/2 \rceil-1$ times as an internal
	 vertex of  $P_2, \ldots, P_{\lceil (g_0+1)/2 \rceil}$, and   $\lceil g_0/2 \rceil+1-\lceil (g_0+1)/2 \rceil$
	 times as an endvertex of $P_1$ ($\lceil g_0/2 \rceil+1-\lceil (g_0+1)/2 \rceil=1$ if and only if $g_0$ is odd, and is 0 otherwise).
	 As a consequence, we have $d_{G_1}(w)=d_G(w)-2\lceil g_0/2 \rceil+\lceil g_0/2 \rceil+1-\lceil (g_0+1)/2 \rceil=\delta$.

	 As edges of $E(P_1)$ were removed from $G$ when getting $G_1$,   we know that $E(P_1)\subseteq E(\overline{G}_1)$.
	 Also, for $P_1$, we have  $w\in V(P_1)$ if and only if $g_0$ is odd, and $w=u_1$ in  this  case.
Since  $Z=\{u_1,u_p\} \cup \{u_{2(i-1)}, u_{2(i-1)+1}\mid i\in [2,\lceil g/2 \rceil]\}$,
	 we know that $ V(u_{2\lceil g/2\rceil}P_1u_{p-1} )\cap (Z\cup V_\delta \cup W) =\emptyset$.
	Since   $u_{2\lceil g/2\rceil}P_1u_{p-1}$  is a path, it has a perfect matching if its order is even, and
	has a matching that covers all but one vertices if its order is odd. Since
	$|V(u_{2\lceil g/2\rceil}P_1u_{p-1})| \ge n-|V_\delta|-|Z|-|W|$ and $u_{2\lceil g/2\rceil}P_1u_{p-1} \subseteq \overline{G_1}$, we then conclude that $\overline{G}_1$ has a matching $M_0$
	 that covers  at least $n-|V_\delta|-|Z|-2$ vertices  of $G_1$ excluding  vertices of $Z \cup V_\delta \cup W$.
	 We consider two cases below regarding the parity of $g$ to complete the proof.
	
	 \medskip
	 {\bf \noindent Case 3.1}: $g$ is even.
	 \medskip
	
	 Then  we have  $d_{G_1}(v)=\delta+1$ for every $v\in Z$ and $d_{G_1}(v)=\delta$ for every $v\in V(G)\setminus Z$.  Thus
	 $\Delta(G_1)=\delta+1$.  Since $g$ is even, we know that $\delta+1$ and $\Delta$ have different parities.

	 If $\Delta$ is odd, then $\delta+1$ is even.   As $G_1$ has  precisely $n-g$ vertices of  degree $\delta$, it follows that $n-g$ is even and so $n$  is even.
	 Also, we know that $n-g\ge \Delta+1-g =\delta+1$.
	 As $n-|V_\delta|-|Z|-2\geq n-\delta-g-1$ (since $|V_\delta|<\eta n$ and $\delta \ge \alpha n$), $\overline{G}_1$ has a matching  $M\subseteq M_0$ covering  exactly $n-g-(\delta+1)$ vertices not including  any vertex of $Z$.
	 Let $G_2$ be obtained from $G_1+M$ by adding a new vertex $x$ and all the edges joining $x$
	 and each vertex from $V(G_1)\setminus (Z\cup V(M))$. The graph $G_2$ is $(\delta+1)$-regular
	 and is a robust $(\nu'/3, \tau^*)$-expander.
	 Thus $G_2$  has a Hamilton decomposition by Theorem~\ref{thm2.13}.
	 Deleting all the edges of $G_2$ incident with $x$ and all the edges of $M$ from the decomposition gives
	 a linear forest decomposition of $G_1$ into $(\delta+1)/2$ linear forests. Thus
	 $\la(G) \le g/2+(\delta+1)/2 =(\Delta+1)/2$, as desired.

	 If $\Delta$ is even, then $\delta+1$ is odd.  If $n$ is odd, then  $n-g-(\delta+1)$ is even. We  proceed
	 exactly the same way above to obtain a $(\delta+1)$-regular graph $G_2$
	 and obtain $\la(G) \le g/2+(\delta+2)/2 =\lceil (\Delta+1) /2\rceil$, as desired.
	 Thus we assume that $n$ is even.  As $n-g-\delta-2 <n-|V_\delta|-|Z|-2$, $\overline{G}_1$  has a matching $M_1 \subseteq M_0$
	 covering  exactly  $n-g-\delta-2$ vertices  of $G_1$   avoiding  vertices of $Z\cup V_\delta \cup W$.
	 We let $G_2$ be obtained from $G_1+M_1$ by adding a new vertex $x$ and all the edges joining $x$
	 and each vertex from $V(G_1)\setminus (Z\cup V(M_1))$.  Then we have $d_{G_2}(x)=\delta+2$ and $d_{G_2}(v)=\delta+1$ for all $v\in V(G_2)$ with $v\ne x$.
	Applying Lemma~\ref{lem:matching-in-complement}, $\overline{G_2}$  has a matching
	excluding the vertex $x$ and covering at least  $n- \frac{n}{n-\delta-1}-3$
	vertices of $G_2$.  Since $\delta \le n-3$ by~\eqref{eqn:V_D-size2} and $g\ge 1$,
	we have $n- \frac{n}{n-\delta-1}-3 \ge n-\delta -2$.
	Thus $\overline{G_2}$ has a matching $M_2$
	covering exactly $n-\delta -2$ vertices excluding the vertex $x$.
	 Now we let
	 $G_3$ be obtained from $G_2+M_2$ by adding a new vertex $y$ and all the edges joining $y$
	 and each vertex from $V(G_2)\setminus (V(M_2) \cup \{x\})$.
	 The graph $G_3$ is  $(\delta+2)$-regular  and is a robust $(\nu'/3, \tau^*)$-expander.
	 Thus $G_3$  has
a Hamilton decomposition by Theorem~\ref{thm2.13}.
	 Deleting all the edges of $G_3$ incident with $x$ or $y$  and all the edges of $M_1\cup M_2$ from the decomposition gives
	 a linear forest decomposition of $G_1$ into $(\delta+2)/2$ linear forests. Thus
	 $\la(G)  \le g/2+(\delta+2)/2 =(\Delta+2)/2 =\lceil(\Delta+1)/2\rceil $, as desired.

	 \medskip
	 {\bf \noindent Case 3.2}: $g$ is odd.
	 \medskip
	
	 Then we have  $d_{G_1}(v)=\delta-1$ for each $v\in V_\Delta \setminus Z$, and $d_{G_1}(v) =\delta$
	 for each $v\in V_\delta \cup W\cup Z$.  Thus $\Delta(G_1)=\delta$.
	
	 If $\delta$ is even, then $\Delta$  is odd.  As $G_1$ has precisely $n-(g+1)-|V_\delta|-|W|$ vertices of  odd degrees, which are all $\delta-1$,
	 we know that  $n-(g+1)-|V_\delta|-|W|$ is even.
	 As $n-|V_\delta|-|Z|-2\geq n-(g+1)-\delta-|V_\delta|-|W|$, $\overline{G}_1$ has a matching  $M\subseteq M_0$ covering  exactly $n-(g+1)-\delta-|V_\delta|-|W|$ vertices not including  any vertex of $Z\cup W\cup V_\delta$.
	 Let $G_2$ be obtained from $G_1+M$ by adding a new vertex $x$ and all the edges joining $x$
	 and each vertex from $V(G_1)\setminus (Z\cup W\cup V_\delta \cup  V(M))$. The graph $G_2$ is $\delta$-regular
	  and is a robust $(\nu'/3, \tau^*)$-expander.
	 Thus $G_2$  has
a Hamilton decomposition by Theorem~\ref{thm2.13}.
	 Deleting all the edges of $G_2$ incident with $x$ and all the edges of $M$ from the decomposition gives
	 a linear forest decomposition of $G_1$ into $\delta/2$ linear forests. Thus
	 $\la(G) \le (g+1)/2+\delta/2 =(\Delta+1)/2$, as desired.

	 Next, consider that $\delta$ is odd and so $\Delta$ is even. Then $G_1$ has  precisely $g+1+|V_\delta|+|W|$ vertices of odd degrees, which are all $\delta$, it follows that $g+1+|V_\delta|+|W|$ is even. Thus if $n$ is even, then
	 $n-(g+1)-|V_\delta|-|W|$ is even.
	 We  proceed
	 exactly the same way above to obtain  a $\delta$-regular graph $G_2$. By Theorem~\ref{thm2.13},
	 we have $\la(G_2) \le (\delta+1)/2$. Thus $\la(G) \le (g+1)/2+(\delta+1)/2 =\lceil (\Delta+1)/2 \rceil$, as desired.

	 Thus we assume that $n$ is odd.  We claim that $p=n-|V_{\delta}|-|W|$ in this case; that is,
	 we cannot have  that $|W|=1$ and $g_0$ is odd.  This is true, since if $|W|=1$ and $g_0$ is odd,
	 then $d_G(w) =\delta+g_0$ is even. Since $\Delta$ is even, this would imply that $|V_\delta|$
	 must be even and so $|W|$ must be even (since $g+1+|V_\delta|+|W|$ is even), a contradiction
	 to $|W|=1$.
	Then as  $g+1+|V_\delta|+|W|$ is even,  we know that  $p=n-|V_{\delta}|-|W|$ is odd. As all the edges in $u_{2\lceil g/2\rceil}P_1u_{p-1}$ are contained in $\overline{G}_1$,
	 and $V(u_{2\lceil g/2\rceil}P_1u_{p-1})=V(G)\setminus (Z\cup V_\delta \cup W)$, we
	 let $M_1$ be a matching of $\overline{G}_1$ covering all the vertices of $V(u_{2\lceil g/2\rceil}P_1u_{p-1})\setminus \{u_{p-1}\}$.  Let $G_2$ be obtained from $G_1$ by adding all the edges of $M_1$
	 and the edge $u_{p-1}u_p$ from $\overline{G}_1$.
	 Then in $G_2$, all vertices other than $u_p$ have degree $\delta$ and $u_p$
	 has degree $\delta+1$.  The graph $\overline{G}_2$ has one single vertex of degree $n-1-(\delta+1)$
	 and all other vertices of degree $n-1-\delta$. By Lemma~\ref{lem:matching-in-complement}, $\overline{G}_2$
	 has a matching covering at least $n-n/(n-\delta)-3$ vertices avoiding  the vertex $u_p$.
	 As $n$ is odd and $\delta$ is odd, we have $\delta \le n-2$. As a consequence,
	 $n-n/(n-\delta)-3 >n-1-(\delta+1)$. Thus $\overline{G}_2$ has a matching $M_2$
	 covering exactly $n-1-(\delta+1)$ vertices excluding the vertex $u_p$.
	 We let $G_3$ be obtained from $G_2+M_2$ by adding a new vertex $x$ and all the edges
	 joining $x$ and vertices from $V(G_2)\setminus (V(M_2) \cup \{u_p\})$.
	 The graph $G_3$ is   $(\delta+1)$-regular  and is a robust $(\nu'/3, \tau^*)$-expander.
	 Thus $G_3$  has
 a Hamilton decomposition by Theorem~\ref{thm2.13}.
	 Deleting all the edges of $G_3$ incident with $x$    and all the edges of $M_1\cup M_2$ from the decomposition gives
	 a linear forest decomposition of $G_1$ into $(\delta+1)/2$ linear forests. Thus
	 $\la(G) \le (g+1)/2+(\delta+1)/2 =(\Delta+2)/2 =\lceil(\Delta+1)/2\rceil $, as desired.

The proof of Theorem~\ref{thm:thm-reduced} is now completed.
\end{proof}

\section*{Acknowledgments}
The authors are deeply grateful to the referee, who carefully read the paper, provided us with highly constructive comments, and guided us toward proving a more general result.

\bibliographystyle{abbrv}
\bibliography{arxiv}

\end{document}